\documentclass[a4paper,12pt,draft]{amsart}

\usepackage[margin=30mm]{geometry}
\usepackage{amssymb, amsmath, amsthm, amscd}

\usepackage[T1]{fontenc}
\usepackage{eucal,mathrsfs,dsfont}
\usepackage{color}
\usepackage{comment}

\renewcommand{\leq}{\leqslant}
\renewcommand{\geq}{\geqslant}
\renewcommand{\le}{\leqslant}
\renewcommand{\ge}{\geqslant}

\definecolor{mno}{rgb}{0.5,0.1,0.5}

\newcommand{\R}{\mathds R}

\newcommand{\Pp}{\mathds P}
\newcommand{\dd}{\mathrm d}
\newcommand{\Ee}{\mathds E}

\newcommand{\capa}{\textup{Cap}}
\newcommand{\I}{\mathds 1}

\def\supp{{\rm supp}}

\newtheorem{theorem}{Theorem}[section]
\newtheorem{lemma}[theorem]{Lemma}
\newtheorem{proposition}[theorem]{Proposition}
\newtheorem{corollary}[theorem]{Corollary}
\numberwithin{equation}{section}

\theoremstyle{definition}

\newtheorem{example}[theorem]{Example}
\newtheorem{remark}[theorem]{Remark}
\newtheorem{assumption}[theorem]{Assumption}
\newtheorem*{ack}{Acknowledgement}

\begin{document}
\allowdisplaybreaks
\title[Rate Functions for Symmetric Markov Processes] {\bfseries
Rate functions for symmetric markov processes via heat kernel}
\author{Yuichi Shiozawa\qquad Jian Wang}

\thanks{\emph{Y. Shiozawa:} Graduate School of Natural Science and Technology,
Department of Environmental and Mathematical Sciences,
Okayama University, Okayama 700-8530, Japan \texttt{shiozawa@ems.okayama-u.ac.jp}}

  \thanks{\emph{J.\ Wang:}
   School of Mathematics and Computer Science, Fujian Normal University, 350007 Fuzhou, P.R. China. \texttt{jianwang@fjnu.edu.cn}}

\date{}

\maketitle

\begin{abstract}
By making full use of heat kernel estimates,
we establish the integral tests on the zero-one laws
of upper and lower bounds
for the sample path ranges of symmetric Markov processes.
In particular, these results concerning on upper rate bounds are applicable
for local and non-local Dirichlet forms,
while lower rate bounds are investigated in both subcritical setting and critical setting.

\medskip

\end{abstract}
\allowdisplaybreaks

\section{Introduction}

We are concerned with the sample path ranges of symmetric Markov
processes generated by regular Dirichlet forms. In particular, we
study the upper and lower bounds of the ranges for all sufficiently
large time. Our purpose in this paper is to establish the integral
tests on the zero-one laws of these bounds by making full use of
heat kernel estimates.

\ \

Let $B=\{B_t\}_{t\geq 0}$ be the Brownian motion on $\R^d$ starting from the origin.
{\it Kolmogorov's test} (see, e.g., \cite[4.12]{IM}) says that,
if $g(t)$ is a positive function on $(0,\infty)$ such that
$g(t)\nearrow\infty$ as $t\rightarrow \infty$,
then
\begin{equation}\label{upper-bm}
\Pp\left(\text{$|B_t| \leq \sqrt{t}g(t)$ for all sufficiently large $t$}\right)=\text{$1$ or $0$}
\end{equation}
according as
$$\int_{1}^{\infty}\frac{1}{t}g(t)^de^{-g(t)^2/2}\,\dd t \ \text{converges or not}.$$
The function $\sqrt{t}g(t)$ is called an {\it upper rate function}
of $B$ if the probability in (\ref{upper-bm}) is $1$.
This function describes the forefront of the Brownian particles.
On the other hand, {\it Dvoretzky and Erd\"os' test} \cite{DE51} (see also \cite[4.12]{IM}) says that
for $d\geq 3$ (i.e., $B$ is transient),
if $h(t)$ is a positive function on $(0,\infty)$ such that $h(t)\searrow 0$ as $t\rightarrow \infty$,
then
\begin{equation}\label{lower-bm}
\Pp\left(\text{$|B_t| \geq  \sqrt{t}h(t)$ for all sufficiently large $t$}\right)=\text{$1$ or $0$}
\end{equation}
according as
$$\int_{1}^{\infty}\frac{1}{t}h(t)^{d-2}\,\dd t \ \text{converges or not}.$$
The function $\sqrt{t}h(t)$ is called a {\it lower rate function} of
$B$ if the probability in (\ref{lower-bm}) is $1$. This function
describes the rear front of the Brownian particles. Even for $d=2$
(i.e., $B$ is recurrent and can not hit any point), the lower rate
function describes how close a particle can go to the origin for all
sufficiently large time. {\it Spitzer's test} \cite{Sp} says that if
$h(t)$ is a positive function on $(0,\infty)$ such that
$h(t)\searrow 0$ as $t\rightarrow \infty$, then (\ref{lower-bm}) is
valid if
$$\int_{1}^{\infty}\frac{1}{t|\log h(t)|}\,\dd t \ \text{converges or not}.$$

These tests on the zero-one 
laws
of rate functions are extended
to symmetric diffusion processes
(see \cite{Ichi,BS} for lower rate functions)
and (symmetric) stable processes on $\R^d$
(see \cite{Kh} for upper rate functions,
and \cite{He, K, T,TW} for lower rate functions).
The full heat kernel estimates are used  especially
for the proof of the zero-probability part.
In this paper, we will get the zero-one 
laws
of rate functions
by developing the approach of the results as mentioned before.
In particular, for the upper rate functions, we use a similar approach of Kim, Kumagai and Wang \cite{KKW}.

\ \

There are a number of results on the rate functions for more general
symmetric Markov processes generated by regular Dirichlet forms.
Grigor'yan and Kelbert \cite{GK} and Grigor'yan \cite{G}
characterized the upper and lower rate functions of the Brownian
motion on a Riemannian manifold in terms of the volume growth rate
(see also \cite{BS}).
These results are refined and extended to Brownian motions on
Riemannian manifolds and symmetric diffusion processes. In
particular, Grigor'yan and Hsu \cite{GHs}, Hsu and Qin \cite{HQ} and
Ouyang \cite{Ou} obtained upper rate functions in terms of the
volume and coefficients growth rates. These results are further
extended to symmetric Markov chains on weighted graphs
(\cite{Hu,HuS}) and symmetric Markov processes with no killing
inside (\cite{S1}). For the lower rate functions, the integral test
by Grigor'yan (\cite{G}) is also extended to symmetric Markov
processes with no killing inside (\cite{S}). Here we note that the
full heat kernel estimates are not needed in general for the results
as mentioned before; however, they are mainly concerned with the
one-probability part. Moreover, in these papers we need to assume
that the distance function belongs locally to the domain of the
Dirichlet forms, but such assumption does not hold for fractals. Our
results in this paper hold for general metric measure spaces
including fractals. It should be emphasized that our results are
applicable to a class of symmetric $\beta$-stable-like  processes with
$\beta\geq 2$.

\ \

The remainder of this paper is arranged as follows. In the next
section, we recall some analytic and probabilistic notions for the
theory of Dirichlet forms and symmetric Hunt processes. Then, by
fully using the heat kernel estimates, we will establish the
integral tests on the zero-one laws of upper and lower bounds for
the sample path ranges of symmetric Markov processes in Sections
\ref{section3} and \ref{section4}, respectively.

\section{Analytic and Probabilistic Notions}

\subsection{Analytic notions}
We recall the notions of Dirichlet forms by following \cite{CF} and \cite{FOT}.
Let $(M,d)$ be a locally compact separable metric space,
and $\mu$ be a positive Radon measure on $M$ with full support.
We write $C(M)$ for the totality of continuous functions on $M$,
and $C_0(M)$ for that of continuous functions on $M$ with compact support.
Let $({\mathcal E}, {\mathcal F})$ be a  Dirichlet form on $L^2(M;\mu)$;
that is, $({\mathcal E}, {\mathcal F})$ is a closed  Markovian symmetric form on $L^2(M;\mu)$.
We assume that $({\mathcal E}, {\mathcal F})$ is {\it regular}, i.e.\
${\mathcal F}\cap C_0(M)$ is dense both in ${\mathcal F}$ with respect to the norm $\sqrt{{\mathcal E}_1}$,
and in $C_0(M)$ with respect to the uniform norm. Here we define
$${\mathcal E}_1(u,u)={\mathcal E}(u,u)+\|u\|_{L^2(M;\mu)}^2, \quad u\in {\mathcal F},$$ where $\|u\|_{L^p(M;\mu)}=(\int_M |u|^p\,\dd \mu)^{1/p}$ for all $p\in(1,\infty)$.

Throughout this paper, we assume that
the Beurling-Deny decomposition of $({\mathcal E}, {\mathcal F})$ (see \cite[Theorem 3.2.1 and Lemma 4.5.4]{FOT}) is given by
\begin{equation*}\label{b-d}
{\mathcal E}(u,v)={\mathcal E}^{(c)}(u,v)
+\iint_{M\times M\setminus{\rm diag}}(u(x)-u(y))(v(x)-v(y))\,n(\dd x, \dd y)
\end{equation*}
for $u,v\in {\mathcal F}\cap C_0(M)$, where
\begin{itemize}
\item $({\mathcal E}^{(c)}, {\mathcal F}\cap C_0(M))$ is a symmetric form enjoying the strong local property
(see \cite[p.120]{FOT} for definition);
\item $n$ is a symmetric positive Radon measure on $M\times M\setminus{\rm diag}$ with
${\rm diag}=\{(x,y)\in M\times M \mid x=y\}$.
\end{itemize}
We call $n$ the {\it jumping measure} associated with $({\mathcal E}, {\mathcal F})$.
We can extend ${\mathcal E}^{(c)}$ uniquely to ${\mathcal F}$.
Furthermore, for $u\in {\mathcal F}$, there exists a positive Radon measure
$\mu_{\langle u\rangle}^c$ on $M$ such that
$${\mathcal E}^{(c)}(u,u)=\frac{1}{2}\mu_{\langle u \rangle}^c(M)$$
(see \cite[p.123]{FOT}).
We call $\mu_{\langle u\rangle}^c$ the {\it local part of the energy measure} of $u$.

Let $\{T_t\}_{t>0}$ be a strongly continuous Markovian semigroup on $L^2(M;\mu)$
associated with $({\mathcal E}, {\mathcal F})$.
Then $T_t$ is extended to $L^{\infty}(M;\mu)$ (see \cite[p.56]{FOT}).
We say that $\{T_t\}_{t>0}$/$({\mathcal E},{\mathcal F})$ is {\it conservative} if
$$T_t1=1, \quad \text{$\mu$-a.e.\ for any $t>0$}.$$
Let
$$S_tf=\int_0^tT_sf\,\dd s, \quad f\in L^2(M;\mu).$$
Here the integral is defined as the Bochner integral in $L^2(M;\mu)$.
We can then extend $T_t$ and $S_t$ on $L^1(M;\mu)\cap L^2(M;\mu)$ to $L^1(M;\mu)$ uniquely.
Let $L_+^1(M;\mu)=\{u\in L^1(M;\mu) \mid \text{$u\geq 0$, $\mu$-a.e.}\}$ and
$$Gf=\lim_{N\rightarrow\infty}S_Nf, \quad f\in L_+^1(M;\mu).$$
We say that $\{T_t\}_{t>0}$/$({\mathcal E},{\mathcal F})$ is {\it recurrent} if
$$Gf=0 \ \text{or} \ \infty, \quad \text{$\mu$-a.e.\ for any  $f\in L_+^1(M;\mu)$},$$
and {\it transient} if
$$Gf<\infty, \quad \text{$\mu$-a.e.\ for any  $f\in L_+^1(M;\mu)$}.$$
A $\mu$-measurable set $A\subset M$ is said to be ($T_t$-){\it invariant} if
$T_t(1_Af)=1_AT_t f$, $\mu$-a.e.\ for any $f\in L^2(M;\mu)$ and $t>0$.
We say that $\{T_t\}_{t>0}$/$({\mathcal E},{\mathcal F})$ is {\it irreducible}
if any invariant set $A$ satisfies either
$\mu(A)=0$ or $\mu(M\setminus A)=0$.
Under the irreducible assumption, $\{T_t\}_{t>0}$ is recurrent or transient, see e.g.\ \cite[Lemma 1.6.4]{FOT}.

Let ${\mathcal F}_e$ be the totality of $\mu$-measurable functions $u$ on $M$ such that
$|u|<\infty$, $\mu$-a.e.\ and
there exists a sequence $\{u_n\} \subset {\mathcal F}$ satisfying that
$\lim_{n\rightarrow\infty}u_n=u$, $\mu$-a.e.\ on $M$ and
$$\lim_{m,n\rightarrow\infty}{\mathcal E}(u_n-u_m, u_n-u_m)=0.$$
This sequence is called an {\it approximating sequence} of $u$.
For any $u\in {\mathcal F}_e$ and its approximating sequence $\{u_n\}$,
the limit
$${\mathcal E}(u,u):=\lim_{n\rightarrow\infty}{\mathcal E}(u_n,u_n)$$
exists and does not depend on the choice of $\{u_n\}$ (see \cite[Theorem 1.5.2]{FOT}).
$({\mathcal F}_e, {\mathcal E})$ is called the {\it extended Dirichlet space} of $({\mathcal E}, {\mathcal F})$
(see \cite[p.41]{FOT}).
In particular, if  $({\mathcal E}, {\mathcal F})$ is transient, then ${\mathcal F}_e$ is complete
with respect to ${\mathcal E}$ (see \cite[Lemma 1.5.5]{FOT}).

We next introduce the notion of capacity.
In what follows, we assume that $({\mathcal E}, {\mathcal F})$ is transient.
Let ${\mathcal O}$ be the totality of open sets in $M$.
For $A\in {\mathcal O}$, define
$${\mathcal L}_A=\{u\in {\mathcal F}_e \mid u\geq 1, \text{$\mu$-a.e.\ on $A$}\}$$
and
$${\capa}_{(0)}(A)=
\begin{cases}
\inf_{u\in {\mathcal L}_A}{\mathcal E}(u,u), & {\mathcal L}_A\ne \emptyset\\
\infty, & {\mathcal L}_A=\emptyset.
\end{cases}
$$
For any  $B\subset M$, we define the {\it  {\rm (}$0$-order{\rm )} capacity} by
$${\capa}_{(0)}(B)=\inf_{A\in {\mathcal O}, \, B\subset A}{\capa}_{(0)}(A).$$
We see from \cite[p.74]{FOT} that if ${\mathcal L}_B\ne\emptyset$, then
there exists a unique element $e_B^{(0)}
 \in {\mathcal L}_B$ such that
$${\capa}_{(0)}(B)={\mathcal E}(e_B^{(0)}, e_B^{(0)}).$$
The function $e_B^{(0)}$ is called the {\it equilibrium potential} of $B$.

For  $A\subset M$, a statement depending on $x\in A$ is said to hold
q.e.\ on $A$ if there exists a set ${\mathcal N}\subset A$ of zero
capacity such that the statement holds for every $x\in A\setminus
{\mathcal N}$. Here q.e.\ is an abbreviation for {\it quasi
everywhere}. A function $u\in {\mathcal F}$ is said to be {\it quasi
continuous} if for any $\varepsilon>0$, there exists $O\in {\mathcal
O}$ with ${\capa}_{(0)}(O)<\varepsilon$ such that $u|_{M\setminus
O}$ is finite continuous, where $u|_{M\setminus O}$ is the
restriction of $u$ on $M\setminus O$. It follows from \cite[Theorem
2.1.7]{FOT} that every $u\in {\mathcal F}_e$ admits its quasi
continuous $\mu$-version $\tilde{u}$.

We say that a positive Radon measure $\nu$ on $M$ is of {\it {\rm (}$0$-order{\rm )} finite energy integral}
($\nu\in S_0^{(0)}$ in notation) if there exists a constant $C>0$ such that
$$\int_M|f|\,\dd\nu\leq C\sqrt{{\mathcal E}(f,f)} \quad \text{for any $f\in {\mathcal F}\cap C_0(M)$}.$$
Then any measure $\nu\in S_0^{(0)}$ charges no set of zero capacity
and associates a unique element $U\nu\in {\mathcal F}_e$,
which is called the {\it {\rm (}$0$-order{\rm )} potential} of $\mu$, such that
$${\mathcal E}(U\nu,v)=\int_M \tilde{v}\,\dd\nu \quad \text{for any $v\in {\mathcal F}_e$}$$
(e.g.\ \cite[p.85]{FOT}).
For any compact set $K$,
there exists a unique measure $\nu_K\in S_0^{(0)}$ with $\supp[\nu_K]\subset K$ such that
$e_K^{(0)}=U\nu_K$ and
\begin{equation*}\label{cap-eq}
{\capa}_{(0)}(K)={\mathcal E}(e_K^{(0)}, e_K^{(0)})=\nu_K(K)
\end{equation*}
($0$-order version of \cite[Lemma 2.2.6]{FOT}).
The measure $\nu_K$ is called the {\it {\rm (}$0$-order{\rm )} equilibrium measure} of $K$.

\subsection{Probabilistic notions}
We write ${\mathcal B}(M)$ for the family of all Borel measurable subsets of $M$.
Let $M_{\Delta}=M\cup\{\Delta\}$ be the one point compactification of $M$
and ${\mathcal B}(M_{\Delta})={\mathcal B}(M)\cup\{B\cup\{\Delta\}: B\in {\mathcal B}(M)\}$.

Let $X=(\{X_t\}_{t\geq 0}, \{{\Pp}^x\}_{x\in M}, \zeta)$ be a
$\mu$-symmetric Hunt process on $M$ generated by $({\mathcal E},
{\mathcal F})$, where $\zeta:=\inf\{t\geq 0 \mid X_t=\Delta\}$ is
the {\it life time} of $X$. Then $({\mathcal E}, {\mathcal F})$ is
 \emph{conservative} if and only if $\Pp^x(\zeta=\infty)=1$ for q.e.\ $x\in
M$  (see \cite[Example 4.5.1]{FOT}). A set $B\subset M$ is called
{\it nearly Borel measurable}, if for any probability measure $\nu$
on $M_{\Delta}$, there exist $B_1, B_2\in {\mathcal B}(M_{\Delta})$
such that $B_1\subset B \subset B_2$ and
$$\Pp^{\nu}(\text{$X_t\in B_2\setminus B_1$ for some $t\geq 0$})=0.$$
We say that a set ${\mathcal N}\subset M$ is {\it properly exceptional},
if ${\mathcal N}$ is nearly Borel measurable such that $m({\mathcal N})=0$
and $M\setminus {\mathcal N}$ is {\it $X$-invariant}; that is,
$$\Pp^x(\text{$X_t\in (M\setminus {\mathcal N})_{\Delta}$ and $X_{t-}\in (M\setminus {\mathcal N})_{\Delta}$ for any  $t>0$})=1,
\quad x\in M\setminus {\mathcal N}.$$ Here $(M\setminus {\mathcal
N})_{\Delta}=(M\setminus {\mathcal N})\cup \{\Delta\}$ and
$X_{t-}=\lim_{s\uparrow t}X_s$.
Note that, by \cite[p.153 and Theorem 4.2.1]{FOT},
any properly exceptional set ${\mathcal N}$ satisfies $\capa_{(0)}({\mathcal N})=0$.

Let $P(t,x,\dd y)$ be the transition function of $X$ given by
$$P(t,x,A)=\Pp^x(X_t\in A), \quad x\in M, \ t\geq 0, \ A\in {\mathcal B}(M).$$
We now impose the following assumption on $X$.

\begin{assumption}\label{abs-cont}\textbf{(Absolute continuity)} \it
There are a properly exceptional Borel set ${\mathcal N}\subset M$ and
a nonnegative symmetric kernel $p(t,x,y)$ on $(0,\infty)\times (M\setminus {\mathcal N})\times (M\setminus {\mathcal N})$
such that
$P(t,x,\dd y)=p(t,x,y)\,\mu(\dd y)$ and
$$p(t+s, x,y)=\int_{M\setminus {\mathcal N}}p(t,x,z)p(s,z,y)\,\mu(\dd z), \quad x,y\in M\setminus {\mathcal N}, \ t,s>0.$$
This kernel is called the {\it heat kernel} associated with $X$.
\end{assumption}

If there exists a positive left continuous function $\psi(t)$ on $(0,\infty)$ such that
$$\|T_tf\|_{\infty}\leq \frac{1}{\psi(t)}\|f\|_1 \quad \text{for any $f\in L^1(M;\mu)$ and $t>0$},$$
then Assumption \ref{abs-cont} holds with
$$p(t,x,y)\le \frac{1}{\psi(t)} \quad \text{for $x,y\in M\setminus {\mathcal N}$ and $t>0$}$$
(see \cite[Theorem 3.1]{BBCK}).
Under Assumption \ref{abs-cont},
we define $p(t,x,y)=0$ for $(x,y)\not\in (M\setminus {\mathcal N})\times (M\setminus {\mathcal N})$ and $t>0$,
so that
\begin{equation*}\label{chap-kol}
p(t+s,x,y)=\int_M p(t,x,z)p(s,z,y)\,\mu(\dd z), \quad x,y\in M, \ t,s>0.
\end{equation*} For more details see \cite[Section 2.2]{GT}.

\begin{remark}\label{recurrence-test}
As we noted in \cite[Remark 2.5]{S}, $({\mathcal E},{\mathcal F})$ is transient,
if
$$\int_1^{\infty}\sup_{y\in M}p(t,x,y)\,\dd t<\infty \quad \text{for any $x\in M$}.$$
Here we also note that $({\mathcal E},{\mathcal F})$ is irreducible,
if $p(t,x,y)$ is strictly positive
for any $(x,y)\in (M\setminus {\mathcal N})\times (M\setminus {\mathcal N})$ and $t>0$.
Moreover, the irreducible Dirichlet form  $({\mathcal E},{\mathcal F})$ is recurrent, if
\begin{equation}\label{rec}
\int_1^{\infty}p(t,x,y)\,\dd t=\infty \quad \text{for any $x,y\in M\setminus {\mathcal N}$}.
\end{equation}
In fact, let $B_b(M)$ be the set of bounded measurable functions on $M$. If $A\subset M$ is an invariant set, then for any $f\in L^2(M;\mu)\cap {B}_b(M)$,
$$T_t(1_Af)(x)=\int_A p(t,x,y)f(y)\,\mu(\dd y)=0, \quad \text{$m$-a.e.\ on $M\setminus A$}.$$
By taking $f$ with $f>0$, $\mu$-a.e.\ on $M$ in the equality above, we get $\mu(A)=0$ or $\mu(M\setminus A)=0$,
which means that  $({\mathcal E},{\mathcal F})$ is irreducible.
Since for any $f\in L^1(M;\mu)\cap {B}_b(M)$ with $f\ge 0$,
$$Gf=\int_0^{\infty}T_tf\,\dd t, \quad \text{$m$-a.e.\ on $M$}$$
and
\begin{equation*}
\begin{split}
\int_0^{\infty}T_tf(x)\,\dd t
&= \int_0^{\infty}\left(\int_M p(t,x,y)f(y)\,\mu(\dd y)\right)\,\dd t\\
&=\int_M\left(\int_0^{\infty}p(t,x,y)\,\dd t\right)f(y)\,\mu(\dd y),
\end{split}
\end{equation*}
(\ref{rec}) implies that the irreducible Dirichlet form $({\mathcal
E},{\mathcal F})$ is recurrent.
\end{remark}

\bigskip

\noindent{{\bf Additional
Notations.}}
 For all $x\in M$ and $r>0$,
$B(x,r)=\left\{y\in M \mid d(y,x)<r\right\}$. For any two positive
measurable functions $f$ and $g$, $f\asymp g$ means that there is a
constant $c>1$ such that $c^{-1} f\le g\le c f$, and $f\lesssim g$
(resp.\ $f\gtrsim g$) means that there is a constant $c>0$ such that
$f\le c g$ (resp.\ $f\ge c g$).

\section{Upper Rate Functions}\label{section3}

Throughout this section, we impose Assumption \ref{abs-cont} on $X$.

\begin{theorem}\label{t1}
\begin{itemize}
 \item[(1)] Let $\mu$ satisfy that $\mu(B(x,r))\lesssim V(r)$ for all $x\in M$ and $r>0$, and let the heat kernel $p(t,x,y)$ satisfy the following upper bound estimate{\rm :}
\begin{equation}\label{upper-1}
p(t,x,y)\le \frac{C}{V(d(x,y))} h\left(\frac{d(x,y)}{\rho(t)}\right)
\end{equation}
for all $t\ge 1$ and $\mu$-almost all $x,y\in M$ with $d(x,y)\ge \rho(t)$.
Here $\rho$ and $V$ are increasing functions on $(0,\infty)$, and $h$ is a decreasing function on $(0,\infty)$ such that
$$V(2r)\le cV(r), \quad r>0 $$ and
\begin{equation}\label{ass-1}
h(\theta r ) \le c_0 h(r), \quad r>1
\end{equation}
with some constants $c>0$, $c_0\in(0,1)$ and $\theta>1$.
If there is an increasing function $\varphi$ on $(1,\infty)$ such that
for some $\varepsilon>0$,
\begin{equation}\label{ass-2}
\int_1^\infty \frac{1}{t} h\left(\frac{\varphi(t)}{2\rho(2(1+\varepsilon)t)} \right)\,\dd t<\infty,
\end{equation}
then for $\mu$-almost all $x\in M$,
\begin{equation}\label{upper-rate}
\Pp^x(d(X_s,x)\le \varphi(s)\textrm{ for all sufficient large }s)=1.
\end{equation}

\item[(2)] Let $\mu$ satisfy that $\mu(B(x,r))\gtrsim V(r)$ for all $x\in M$ and
$r>0$, and let the heat kernel $p(t,x,y)$ satisfy the following
lower bound estimate{\rm :}
\begin{equation*}\label{lower-1}
p(t,x,y)\ge  \frac{C}{V(d(x,y))} h\left(\frac{d(x,y)}{\rho(t)}\right)
\end{equation*}
for all $t\ge 1$ and $\mu$-almost all $x,y\in M$ with $d(x,y)\ge \rho(t)$.
Here $\rho$ and $V$ are increasing functions on $(0,\infty)$, and $h$ is a decreasing function on $(0,\infty)$ such that
\begin{equation}\label{ass-v}
c_1 \Big(\frac Rr\Big)^{d_1} \leq \frac{V(R)}{V (r)},\quad  0<r<R<\infty
\end{equation}
and
\begin{equation}\label{ass-3}
c_0  h(2r)\ge h(r),\quad r>1
\end{equation}
for some constants $c_1,d_1>0$ and $c_0>1.$
If there is an increasing function $\varphi$ on $(1,\infty)$ such that
\begin{equation}\label{ass-4}
\int_1^\infty \frac{1}{t} h\left(\frac{2\varphi(4t)}{\rho(t)} \right)\,\dd t=\infty,
\end{equation}
then for $\mu$-almost all $x\in M$,
$$\Pp^x(d(X_s,x)\le \varphi(s)\textrm{ for all sufficient large }s)=0.$$
\end{itemize}
\end{theorem}

The function $\varphi(s)$ satisfying (\ref{upper-rate}) is called the {\it upper rate function} of $X$.
\begin{proof}[Proof of Theorem $\ref{t1}$]  Our approach here is similar to that in \cite[Section 3]{KKW}.

(1) \
Let us first check that there exists a constant $c_1>0$ such that for $\mu$-almost all $x\in M$, $r>0$ and $t\ge 1$,
\begin{equation}\label{tail}
\Pp^x(d(X_t,x)\geq r)=\int_{B(x,r)^c}p(t,x,z)\,\mu(\dd z)\le c_1h\left(\frac{r}{\rho(t)} \right).
\end{equation}
If $r\le \rho(t)$, then we find by the decreasing property of $h$
that \eqref{tail} holds with $c_1=\frac{1}{h(1)}.$ Next, we suppose
that $r>\rho(t)$. According to \eqref{upper-1}, for all $t\ge 1$ and
almost all  $x,z\in M$ with $d(x,z)\ge s>\rho(t),$
$$p(t,x,z)\le \frac{C}{V(s)}h\left(\frac{s}{\rho(t)}\right).$$
Then, for all $t\ge 1$, $r>\rho(t)$ and $\mu$-almost all  $x\in M$,
\begin{align*}
\int_{B(x,r)^c}p(t,x,z)\,\mu(\dd z)&\le \sum_{k=0}^\infty \int_{B(x,\theta^{k+1}r)\setminus B(x,\theta^k r)} p(t,x,z)\,\mu(\dd z)\\
&\le c_2\sum_{k=0}^\infty \frac{1}{V(\theta^kr)} h\left( \frac{\theta^kr}{\rho(t)}\right) V(\theta^{k+1}r) \\
&\le c_3  h\left( \frac{r}{\rho(t)}\right)\sum_{k=0}^\infty  c_0^{k}\\
&\le c_4 h\left( \frac{r}{\rho(t)}\right),\end{align*}
where in the third inequality we have used the assumption on $V$ and \eqref{ass-1}.

Now, let
$$\tau_{B(x,r)}=\inf\Big\{ t>0: X_t\notin B(x,r)\Big\}.$$
Then
$$
\Pp^x(\tau_{B(x,r)}\le t)\le \Pp^x(\tau_{B(x,r)}\le t, d(X_{2t},x)\le r/2)+\Pp^x(d(X_{2t},x)\ge r/2).
$$
We have, by (\ref{tail}),
$$\Pp^x(d(X_{2t},x)\ge r/2)\leq c_1h\left(\frac{r}{2\rho(2t)}\right).$$
By (\ref{tail}) and the strong Markov property of $X$, for $\mu$-almost all $x\in M$, $t\ge1$ and $r>0$,
 \begin{equation*}
\begin{split}
\Pp^x\Big(\tau_{B(x,r)}\le t, d(X_{2t},x)\le r/2\Big)
   \le &\Pp^x\Big(\tau_{B(x,r)}\le t, d(X_{2t},X_{\tau_{B(x,r)}})\ge r/2\Big)\\
   \le &\sup_{s\le t, d(z,x)\ge r} \Pp^z (d(X_{2t-s},z)\ge r/2)\\
   \le &\sup_{s\le t, d(z,x)\ge r} c_1h\left(\frac{r}{2\rho(2t-s)}\right)\\
   \le &c_1 h\left( \frac{r}{2\rho(2t)}\right),
\end{split}
\end{equation*}
where the last inequality follows from the monotonicity conditions on $\rho$ and $h$.
Hence
\begin{equation*}\label{exit}
 \Pp^x(\tau_{B(x,r)}\le t) \le 2c_1 h\left( \frac{r}{2\rho(2t)}\right).
\end{equation*}
In particular, setting $t_k=(1+\varepsilon)^{k/2}$, we have for all $k\ge 2$,
\begin{equation*}
\begin{split}
\Pp^x(d(X_s, x)\ge &\varphi(s)\textrm{ for some }s\in (t_{k-1},t_k])\\
\le& \Pp^x\left(\sup_{s\in(t_{k-1},t_k]}d(X_s,x)\ge \varphi(t_{k-1})\right)
\le \Pp^x(\tau_{B(x,\varphi(t_{k-1}))}\leq t_k) \\
\le& 2c_1 h\left( \frac{\varphi(t_{k-1})}{2\rho({2t_k})}\right)=2c_1 h\left( \frac{\varphi(t_{k-1})}{2\rho({2(1+\varepsilon)t_{k-2}})}\right).
\end{split}
 \end{equation*}
This, along with \eqref{ass-2} and the Borel-Cantelli lemma, yields the first desired assertion.

(2) We replace $\varphi(4t)$ with $\varphi(t)$ in the proof. First, choose $r_0\ge2$ such that $r_0^{-d_1}<c_1$, where $d_1$ and $c_1$ are constants given in \eqref{ass-v}. By
\eqref{ass-v} and \eqref{ass-3}, there exists a constant $c_0^*\in (0,1)$ such that for any $s\ge \rho(t)$, we have
\begin{align*}
\int_{r\ge s} \frac{1}{V(r) }h\left( \frac{r}{\rho(t)}\right)\,\dd V(r)
&=\sum_{k=0}^\infty \int_{r\in [r_0^ks, r_0^{k+1}s)}\frac{1}{V(r)}h\left( \frac{r}{\rho(t)}\right)\,\dd V(r)\\
&\ge \sum_{k=0}^\infty \frac{V(r_0^{k+1}s)-V(r_0^{k}s)}{V(r_0^{k+1}s)}h\left( \frac{r_0^{k+1}s}{\rho(t)}\right)\\
&\ge \left(1-\frac{1}{c_1r_0^{d_1}}\right)\sum_{k=0}^\infty h\left( \frac{r_0^{k+1}s}{\rho(t)}\right)\\
&\ge c_2 h\left( \frac{s}{\rho(t)}\right)\sum_{k=0}^\infty {c_0^*}^{1+k} \\
&=: c_3 h\left( \frac{s}{\rho(t)}\right),
\end{align*}
which shows that
\begin{equation}\label{sup-2-1}
\inf_{t\ge1}\int_{r\ge \rho(t)}\frac{1}{V(r) }h\left( \frac{r}{\rho(t)}\right)\,\dd V(r)\geq c_3h(1)>0.
\end{equation}

For any $k\ge 1$, set  $t_k=2^k$ and
$$B_k=\{ d(X_{2^{k+1}},X_{2^k}) \ge 2\varphi (2^{k-1}) \vee \rho(2^{k})\}.$$
Then for every $x\in M \setminus \mathcal{N}$ and $k\ge 1$, by the Markov property,
\begin{align*}
\Pp^x (B_k|\mathscr{F}_{t_{k}})\ge &\inf_{z\in M\setminus{\mathcal N}}\Pp^z\Big(d(X_{2^k}, z)\ge
2\varphi (2^{k-1}) \vee \rho(2^{k})\Big)\\
\ge& c_4\int_{r\ge 2\varphi (2^{k-1}) \vee \rho(2^{k})}  \frac{1}{V(r)}h\left( \frac{r}{\rho(2^k)}\right)\,\dd V(r).
\end{align*}

If there exist infinitely many $k\ge 1$ such that $2\varphi(2^{k-1})\le \rho(2^{k})$,
then, by \eqref{sup-2-1}, for such $k\ge 1$,
\begin{align*}
\Pp^x (B_k|\mathscr{F}_{t_{k}})
\ge & c_4\int_{r\ge\rho(2^{k})} \frac{1}{V(r)}h\left( \frac{r}{\rho(2^k)}\right)\,\dd V(r)\\
\ge & c_4\inf_{t\ge 1} \int_{r\ge\rho(t)}\frac{1}{V(r)}h\left( \frac{r}{\rho(t)}\right)\,\dd V(r)=:c_5>0
\end{align*}
and so
\begin{equation}\label{sup-2-3}
\sum_{k=1}^\infty \Pp^x (B_k|\mathscr{F}_{t_{k}})=\infty.
\end{equation}
On the other hand, if there is $k_0\ge1$ such that $2\varphi(2^{k-1})>\rho(2^{k})$ for all $k\ge k_0$, then
$$
\Pp^x (B_k|\mathscr{F}_{t_{k}})
\ge c_4\int_{r\ge 2\varphi(2^{k-1})} \frac{1}{V(r)}h\left( \frac{r}{\rho(2^k)}\right)\,\dd V(r)
\geq c_5h\left(\frac{2\varphi(2^{k-1})}{\rho(2^k)}\right).
$$
Therefore, by (\ref{ass-4}) and the second Borel-Cantelli lemma, $\Pp^x(\limsup B_n)=1$. This implies that for infinitely many $k\ge 1$,
$$
d(X_{2^{k+1}},x)\ge \frac{1}{2}( 2\varphi(2^{k-1})  \vee \rho(2^{k}))\ge {\varphi\left(\frac{2^{k+1}}{4}\right)}
$$
or
$$
d(X_{2^k}, x)\ge \frac{1}{2}( 2\varphi(2^{k-1}) \vee \rho(2^{k}))\ge {\varphi\left(\frac{2^{k}}{4}\right)}.
$$
In particular,
$$
\limsup_{t\to \infty}\frac{d(X_t,x)}{\varphi\big(\frac{t}{4}\big) }
\ge \limsup_{k\to \infty} \frac{d(X_{2^{k}},x)}{\varphi\big(\frac{2^k}{4}\big)}
\ge1,
$$
which yields the desired assertion.
 \end{proof}

At the end of this section, we present two typical examples as
applications of Theorem \ref{t1}.

\begin{example}\label{exm-3.2}\begin{itemize}
\item[(1)] \emph{{\rm \textbf{(Non-local Dirichlet forms)}}}\,\emph{
Let $V(r)=r^\alpha$, $\rho(t)=t^{1/\beta}$ and $h(s)=s^{-\beta}$ with $\alpha, \beta>0$.
Suppose that the heat kernel $p(t,x,y)$ satisfies the following estimate{\rm:}
there is a constant $c>0$ such that for almost all $x,y\in M$ and $t\ge 1$ with $d(x,y)\ge t^{1/\beta}$,
\begin{equation}\label{hk-stable-like} p(t,x,y)\asymp t^{-\alpha/\beta} \left( 1+\frac{d(x,y)}{t^{1/\beta}}\right)^{-\alpha-\beta}\asymp  \frac 1{t^{\alpha/\beta}}\wedge\frac{t}{d(x,y)^{\alpha+\beta}}.\end{equation}
Then, we have the following two statements{\rm:}
\begin{itemize}
\item[(i)] If there is an increasing function $\varphi$ on $(1,\infty)$ such that
$$\int_1^\infty \varphi(t)^{-\beta}\,\dd t<\infty,$$ then for $\mu$-almost all $x\in M$,
$$
\Pp^x(d(X_s,x)\le \varphi(s)\textrm{ for all sufficient large }s)=1.$$
\item[(ii)]If there is an increasing function $\varphi$ on $(1,\infty)$ such that
$$\int_1^\infty \varphi(t)^{-\beta}\,\dd t=\infty,$$ then for $\mu$-almost all $x\in M$,
$$
\Pp^x(d(X_s,x)\le \varphi(s)\textrm{ for all sufficient large }s)=0.$$
\end{itemize} Consequently, for any constant $c>0$,
the function $$\varphi(t)=ct^{1/\beta}(\log t)^{(1+\varepsilon)/\beta},\quad \varepsilon\in \R$$
is an upper rate function of $X$, if and only if $\varepsilon>0$.}

Let $J(x,y)$ be a positive and symmetric measurable function on $M\times M\setminus {\rm diag}$ such that
$$J(x,y)\asymp \frac{1}{d(x,y)^{\alpha+\beta}} \quad \text{for any $(x,y)\in M\times M\setminus {\rm diag}$},$$
and let $\mu$ satisfy that $\mu(B(x,r))\asymp r^{\alpha}$ for any $x\in M$ and $r>0$.
We see from \cite{CK, CK1} that, if $0<\beta<2$ and the Dirichlet form $({\mathcal E}, {\mathcal F})$ is given by
\begin{equation*}
\begin{split}
{\mathcal E}(u,v)&=\iint_{M\times M\setminus {\rm diag}}(u(x)-u(y))(v(x)-v(y))J(x,y)\,\mu(\dd x)\,\mu(\dd y),\\
{\mathcal F}&=\left\{u\in L^2(M;\mu)\mid {\mathcal E}(u,u)<\infty\right\},
\end{split}
\end{equation*}
then the associated heat kernel satisfies the condition (\ref{hk-stable-like}), and
the statements (i), (ii) hold.
Here we emphasize that these assertions are valid even for $\beta\geq 2$
if we consider a class of subordinated fractional diffusion processes
(see, e.g., \cite{Ku, BSS}).

\item[(2)] \emph{{\rm \textbf{(Local Dirichlet forms)}}}\,\emph{ Assume that $V(r)=r^\alpha$, $\rho(t)=t^{1/\beta}$
and $h(s)=\exp(-c_0s^{\beta/(\beta-1)})$ with $\alpha, c_0>0$ and $\beta>1$.
Suppose that the heat kernel $p(t,x,y)$ satisfies the following estimate{\rm:}
there exists a constant $c>0$ such that for almost $x,y\in M$ and $t\ge 1$ with $d(x,y)\ge t^{1/\beta}$,
\begin{equation}\label{hk-diffusion}p(t,x,y)\asymp t^{-\alpha/\beta} \exp\left( -c_0\left(\frac{d(x,y)}{t^{1/\beta}}\right)^{\beta/(\beta-1)}\right).\end{equation}
Then, we have
\begin{itemize}
\item[(i)]
For any $\eta> 2^{1+1/\beta} c_0^{-(\beta-1)/\beta} $ and for $\mu$-almost all $x\in M$,
$$
\Pp^x\left(d(X_s,x)\le  \eta s^{1/\beta} \big(\log\log s\big)^{(\beta-1)/\beta}\textrm{ for all sufficient large }s\right)=1.$$
\item[(ii)]For any $\eta< 2^{-1-2/\beta} c_0^{-(\beta-1)/\beta} $ and for $\mu$-almost all $x\in M$,
$$
\Pp^x\left(d(X_s,x)\le  \eta s^{1/\beta} \big(\log\log (s/4)\big)^{(\beta-1)/\beta}\textrm{ for all sufficient large }s\right)=0.$$
\end{itemize}}
\end{itemize}
\end{example}

By the condition \eqref{hk-diffusion}, we can get that for almost any $x,y\in M$, $t\ge 1$ with $d(x,y)\ge t^{1/\beta}$, and any constant $\varepsilon>0$,
\begin{align*}  \frac{c_1}{d(x,y)^{\alpha}} \exp\left( -c_0\left(\frac{d(x,y)}{t^{1/\beta}}\right)^{\beta/(\beta-1)}\right)
 &\le p(t,x,y)\\
 &\le  \frac{c_2}{d(x,y)^{\alpha}}  \exp\left( -(c_0\!-\!\varepsilon)\left(\frac{d(x,y)}{t^{1/\beta}}\right)^{\beta/(\beta-1)}\right),\end{align*} where $c_1,c_2=c_2(\alpha,\beta,\varepsilon)$ are some positive constants. Then, the desired statement follows from Theorem \ref{t1} and the estimate above.

\section{Lower Rate Functions}\label{section4}
In this section, we assume that the heat kernel $p(t,x,y)$ satisfies the following two-sided estimate: for $\mu$-almost all $x,y\in M$ and $t>0$,
\begin{equation}\label{hkestimate}
p(t,x,y) \asymp \left( \frac{1}{V(\phi^{-1}(t))}  \wedge \frac{t}{V(d(x,y)) \phi(d(x,y))} \right),
\end{equation}
where $V$ and $\phi$ are increasing functions such that there exist
constants $c_i,d_i>0$ $(i=1,2,3,4)$ satisfying
\begin{equation}\label{volume-cond}
  c_1 \Big(\frac Rr\Big)^{d_1} \leq \frac{V(R)}{V (r)} \
\leq \ c_2 \Big(\frac Rr\Big)^{d_2}, \quad 0<r<R<\infty
\end{equation}
and
\begin{equation}\label{scaling-cond}
c_3\Big(\frac{R}{r}\Big)^{d_3}
\le \frac{\phi(R)}{\phi(r)}
\le c_4\Big(\frac{R}{r}\Big)^{d_4},\quad  0<r<R<\infty.
\end{equation}

\begin{remark}
Under the setting above, the heat kernel estimate \eqref{hkestimate}
indeed holds for all $x,y\in M$ and $t>0$, i.e.\
$\mathcal{N}=\emptyset$. We also have for all $x\in M$ and $r>0$,
$$\mu(B(x,r))\asymp V(r).$$
See \cite[Section 2]{KKW}.
\end{remark}

On the other hand, under the setting in this section we have the
following statement for heat kernels.

\begin{lemma}\label{comp-heat}
For all $x,y, z\in M$ and $t>0$ with $d(x,y)\leq \phi^{-1}(t)$,
$$p(t,x,z)\asymp p(t,y,z).$$
\end{lemma}
\begin{proof}
We first assume that $d(x,z)\leq \phi^{-1}(t)$ and $d(y,z)\leq \phi^{-1}(t)$.
Then (\ref{hkestimate}) implies that
$$p(t,x,z)\asymp p(t,y,z)\asymp \frac{1}{V(\phi^{-1}(t))}.$$
We next assume that $d(x,z)\leq \phi^{-1}(t)$ and $d(y,z)\geq \phi^{-1}(t)$.
Since
$$d(y,z)\leq d(y,x)+d(x,z)\leq 2\phi^{-1}(t),$$
we have, by (\ref{hkestimate}), (\ref{volume-cond}) and (\ref{scaling-cond}),
$$p(t,y,z)\asymp \frac{t}{V(d(y,z))\phi(d(y,z))}\asymp \frac{1}{V(\phi^{-1}(t))}\asymp p(t,x,z).$$
This relation is valid even if $d(x,z)\geq \phi^{-1}(t)$ and $d(y,z)\leq \phi^{-1}(t)$.
We finally assume that  $d(x,z)\geq \phi^{-1}(t)$ and $d(y,z)\geq \phi^{-1}(t)$.
Then
$$d(y,z)\leq d(y,x)+d(x,z)\leq \phi^{-1}(t)+d(x,z)\leq 2d(x,z).$$
By interchanging $x$ and $y$, we get $d(x,z)\asymp d(y,z)$.
Hence, by  (\ref{hkestimate}), (\ref{volume-cond}) and (\ref{scaling-cond}),
$$p(t,x,z)\asymp \frac{t}{V(d(x,z))\phi(d(x,z))}\asymp \frac{t}{V(d(y,z))\phi(d(y,z))}\asymp p(t,y,z).$$
Therefore, the proof is  complete.
\end{proof}

\subsection{Subcritical case}
In this subsection, we impose the following assumption:

\begin{assumption}\label{assmp-1}
(i) \emph{The constants $d_i \ (i=1,4)$ in $(\ref{volume-cond})$ and $(\ref{scaling-cond})$
satisfy
$$d_1>d_4.$$}
(ii)  \emph{For any $x\in M$ and $r>0$, $B(x,r)$ is  relatively compact in $M$. }
\end{assumption}
Our main result in this subsection is as follows:
\begin{theorem} \label{lower-rate}
Let Assumption $\ref{assmp-1}$ hold. Let $g(t)$ be a strictly
positive function on $(1,\infty)$ such that $g(t)\searrow 0$ as
$t\rightarrow\infty$. If the function $\varphi(t):=\phi^{-1}(t)g(t)$
satisfies
$$\int_1^\infty \frac{ V(\varphi(t))}{\phi (\varphi(t)) V(\phi^{-1}(t))}\, \dd t <\infty \quad (\text{{\it resp.} $=\infty$}),$$
then for all $x\in M$,
$$ \Pp^x(d(X_s,x)\ge \varphi(s)\textrm{ for all sufficiently large }s)=1 \quad (\text{{\it resp.} $=0$}).$$
\end{theorem}

We first present some consequences of Assumption \ref{assmp-1} (i).
\begin{lemma}
Under Assumption $\ref{assmp-1}$ {\rm (i)},
\begin{equation}\label{ass-v-phi} \int_t^\infty\frac{1}{V(\phi^{-1}(s))}\,\dd s \asymp \frac{t}{V(\phi^{-1}(t))},\quad t>0.
\end{equation}
In particular,
\begin{itemize}
\item[(i)] the process $X$ is transient.
\item[(ii)] there is a constant $c>0$ such that for any $r_1<r_2$,
\begin{equation}\label{ass-v-phi-1}\frac{\phi(r_2)}{V(r_2)}\le c\frac{\phi(r_1)}{V(r_1)}.\end{equation}\end{itemize}
\end{lemma}
\begin{proof} For any $t>0$,
\begin{align*}
\int_t^\infty \frac{1}{V(\phi^{-1}(s))}\,\dd s =&\sum_{k=0}^\infty \int_{2^kt}^{2^{k+1}t}\frac{1}{V(\phi^{-1}(s))}\,\dd s
\le \sum_{k=0}^\infty \frac{2^kt}{V(\phi^{-1}(2^kt))}\\
\le & \frac{c_4^{d_1/d_4}}{c_1}\frac{t}{V(\phi^{-1}(t))}\sum_{k=0}^\infty  2^{-k(d_1/d_4-1)}\\
\asymp &  \frac{t}{V(\phi^{-1}(t))} ,
\end{align*}
where in the third equality we have used the fact that
\begin{equation}\label{vp}\frac{V(\phi^{-1}(R))}{V(\phi^{-1}(r))}\ge c_0 \left( \frac{R}{r}\right)^{ d_1/d_4}.
\end{equation}
On the other hand, for any $t>0$,
$$\int_t^\infty\frac{1}{V(\phi^{-1}(s))}\,\dd s \ge \int_t^{2t}\frac{1}{V(\phi^{-1}(s))}\,\dd s
\ge  \frac{t}{V(\phi^{-1}(2t))}\gtrsim \frac{t}{V(\phi^{-1}(t))}.$$
We have proved \eqref{ass-v-phi}.

For all $x\in M$,
$$\int_1^{\infty} \sup_{y\in M}p(t,x,y)\,\dd t
\lesssim \int_1^\infty  \frac{1}{V(\phi^{-1}(t))}\,\dd t<\infty,$$
which along with Remark \ref{recurrence-test} yields the assertion (i).

According to \eqref{ass-v-phi}, for any $r_1<r_2$,
$$\frac{\phi(r_2)}{V(r_2)}\le c'\int_{\phi(r_2)}^\infty \frac{1}{V(\phi^{-1}(s))}\,\dd s \le c'\int_{\phi(r_1)}^\infty  \frac{1}{V(\phi^{-1}(s))}\,\dd s \le c''\frac{\phi(r_1)}{V(r_1)}, $$ which yields the assertion (ii).
\end{proof}

In the remainder of this subsection, we always assume that Assumption $\ref{assmp-1}$ holds.
For any $x,y\in M$, let $u(x,y)$ be the Green function of the associated process $X$:
$$u(x,y)=\int_0^\infty p(t,x,y)\,\dd t.$$
\begin{lemma} \label{lemma-1}
For any $x,y\in M$,
$$u(x,y)\asymp \frac{\phi(d(x,y))}{V(d(x,y))}.$$
\end{lemma}
\begin{proof} For any $x,y\in M$,
\begin{align*}
u(x,y)&=\int_0^\infty p(t,x,y)\,\dd t\\
&\asymp\bigg[\int_{t\le \phi(d(x,y))} \frac{t}{V(d(x,y)) \phi(d(x,y))}\,\dd t+ \int_{t\ge \phi(d(x,y))} \frac{1}{V(\phi^{-1}(t))}\,\dd t\bigg]\\
&\asymp\frac{\phi(d(x,y)) }{V(d(x,y))},
\end{align*}
where in the last inequality we have used
\eqref{ass-v-phi}.
     \end{proof}

For every compact set $K$ of $M$, define
$$\sigma_K=\inf\{t>0: X_t\in K\}.$$
If $({\mathcal E},{\mathcal F})$ is transient and Assumption \ref{abs-cont} holds,
then
\begin{equation}\label{hunt}
\Pp^x(\sigma_K<\infty)=\int_K u(x,y)\,\nu_K(\dd y), \quad \text{$\mu$-a.e.\ $x\in M$},
\end{equation}
where $\nu_K$ is the associated equilibrium measure of $K$. See \cite[Chapter VI]{BulG} or \cite{S} for details.
Using this with Lemma \ref{lemma-1}, we get the following lower bound of capacity.

\begin{lemma} \label{lemma-2}
For any $x_0\in M$ and $r>0$,
$$\capa_{(0)}(\overline{B(x_0,r)})\gtrsim\frac{V(r)}{\phi(r)}.$$
\end{lemma}
\begin{proof}
We first note that $\overline{B(x_0,r)}$ is compact thanks to Assumption \ref{assmp-1} (ii).
Then, by the 0-order version of \cite[Theorem 2.1.5 and Theorem 4.3.3]{FOT}.
$$\Pp^x(\sigma_{\overline{B(x_0,r)}}<\infty)=1, \quad \text{$\mu$-a.e.\ $x\in {\overline{B(x_0,r)}}$}. $$
Therefore, \eqref{hunt} implies that
 \begin{align*}
 1=& \frac{1}{\mu(\overline{B(x_0,r)})}\int_{\overline{B(x_0,r)}} \Pp^y(\sigma_{\overline{B(x_0,r)}}<\infty)\,\mu(\dd y)\\
 =&\frac{1}{\mu(\overline{B(x_0,r)})}\int_{\overline{B(x_0,r)}}\int_{\overline{B(x_0,r)}} u(y,z)\,\mu(\dd y)\,\nu_{\overline{B(x_0,r)}}(\dd z).
 \end{align*}

According to the equality above and Lemma \ref{lemma-1}, we find
that
\begin{align*}
1\le &c_1\frac{1}{\mu(\overline{B(x_0,r)})}
\int_{\overline{B(x_0,r)}}\int_{\overline{B(x_0,r)}} \frac{\phi(d(y,z))}{V(d(y,z))}\,\mu(\dd y)\,\nu_{\overline{B(x_0,r)}}(\dd z)\\
\le&c_2\frac{1}{\mu(\overline{B(x_0,r)})}
\int_0^{2r} \frac{\phi(s)}{V(s)}\,\dd V(s) \int_{\overline{B(x_0,r)}}\,\nu_{\overline{B(x_0,r)}}(\dd z)\\
\le&\frac{ c_3 \phi(r)}{V(r)}  \capa_{(0)}(\overline{B(x_0,r)})
,\end{align*} where in the second and third inequalities we have
used the facts that
$$\mu(B(x,r))\asymp V(r),\quad x\in M, r>0$$
and
$$\int_0^{t} \frac{\phi(s)}{V(s)}\,\dd V(s)\asymp \phi(t),\quad t>0,$$
respectively.
Therefore,
$$  \capa_{(0)}(\overline{B(x_0,r)})\ge c_3^{-1} \frac{V(r)}{ \phi(r)},$$
which completes the proof.
\end{proof}

\begin{remark}
(i) \ We can also get the upper bound of $\capa_{(0)}(\overline{B(x_0,r)})$, i.e.\
for any $x_0\in M$ and $r>0$,
$$\capa_{(0)}(\overline{B(x_0,r)})\lesssim\frac{V(r)}{\phi(r)}.$$
To prove this, we adopt the following equivalent notion for the capacity.
For any Borel sets $B\subset A \subset M$, define
$$
\capa_{(0)}(B,A)=\inf\Big\{\mathcal{E}(u,u):u\in \mathcal{F},  u|_B=1, u|_{A^c}=0\Big\}.
$$ In particular, by the $0$-order version of \cite[Lemma 2.17]{FOT},
$\capa_{(0)}(B)=\capa_{(0)}(B, M)$. On the other hand, it is clear that
$$\capa_{(0)}(\overline{B(x_0,r)})\le \capa_{(0)}(\overline{B(x_0,r)}, \overline{B(x_0,2r)}).$$
Furthermore, according to \cite{CKW},
there exists a constant $c_4>0$ such that for all $x_0\in M$ and $r>0$,
$$\capa_{(0)}(\overline{B(x_0,r)}, \overline{B(x_0,2r)})\le  c_4 \frac{V(r)}{ \phi(r)},$$
which along with the inequality above yields the required upper bound.

\noindent
(ii) \ Under some strong conditions on the Dirichlet form,
we can give another proof of  Lemma \ref{lemma-2}
by following the argument of Fukushima-Uemura in \cite[Section 3]{FU}.
Let $J(x,y)$ be a positive and symmetric measurable function
on $M\times M\setminus {\rm diag}$ such that
for some constants $\alpha,\beta>0$,
$$J(x,y)\asymp \frac{1}{d(x,y)^{\alpha+\beta}} \quad
\text{for any $(x,y)\in M\times M\setminus {\rm diag}$},$$
and let $\mu$ satisfy that $\mu(B(x,r))\asymp r^{\alpha}$ for any $x\in M$ and $r>0$.
If the Dirichlet form $({\mathcal E}, {\mathcal F})$ is given by
\begin{equation*}
\begin{split}
{\mathcal E}(u,v)&=\iint_{M\times M\setminus {\rm diag}}(u(x)-u(y))(v(x)-v(y))J(x,y)\,\mu(\dd x)\,\mu(\dd y)\\
{\mathcal F}&=\left\{u\in L^2(M;\mu)\mid {\mathcal E}(u,u)<\infty\right\},
\end{split}
\end{equation*}
then as we mentioned in the remark below Example \ref{exm-3.2} (i),
$$p(t,x,y)\lesssim \frac{1}{t^{\alpha/\beta}}$$
for any $x,y\in M$ and $t>0$. We further assume that $\alpha>\beta$.
Then, $({\mathcal E}, \mathcal{F})$ is transient, and we have the Sobolev inequality
$$\|u\|_{L^{2\alpha/(\alpha-\beta)}(M;\mu)}^2\leq C{\mathcal E}(u,u), \quad u\in {\mathcal F}$$
for some constant $C>0$ (see \cite[Lemma 2.1.2 and Theorem 2.4.2]{D} or \cite[Theorem 1]{V}).
This inequality implies that
$$C\capa_{(0)}(\overline{B}(x,r))
\geq \mu(\overline{B}(x,r))^{(\alpha-\beta)/\alpha}\asymp
r^{\alpha-\beta}.$$ Therefore, we get the assertion of Lemma
\ref{lemma-2}.
\end{remark}

\ \

Next, we turn to consider the one-probability statement in Theorem \ref{lower-rate}.
For any $x\in M$ and Borel set $K$, define
$$\gamma_{x,K}(A)=\Pp^x(X_{\sigma_K}\in A), \quad A\in \mathcal{B}(M).$$
Since the trajectories of the process $X$ are right continuous,
for any closed set $K$,  $X_{\sigma_K}\in K$ and
\begin{equation}\label{o-1}
\gamma_{x,K}(K)= \Pp^x(X_t\in K\textrm{ for some }t>0).
\end{equation}
Following \cite[Section 8.2]{GK1}, we get
\begin{lemma}
For any closed set $K\subset M$, and for all $x\notin K$ and $y\in K$,
\begin{equation}\label{o-2}
u(x,y)=\int_K u(z,y)\,\gamma_{x,K}(\dd z).
\end{equation}
In particular,
\begin{equation}\label{o-3}
\Pp^x(X_t\in K\textrm{ for some }t>0)\le \frac{u(x,y)}{\inf_{z\in K} u(z,y)}.
\end{equation}
\end{lemma}

\begin{proof}
For any $x\notin K$ and $y\in K$, let $\pi_{x,K}(\dd s,\dd z)$ be the joint law of $(\sigma_K, X_{\sigma_K})$ with the starting point $X_0=x.$
Then, by the strong Markov property,
$$p(t,x,y)= \Ee^x\Big(\I_{\{\sigma_K<t\}} p(t-\sigma_K, X_{\sigma_K}, y)\Big)=\int_K\int_0^t p(t-s,z,y)\,\pi_{x,K}(\dd s,\dd z).$$
Integrating both sides with respect to $t$, we get that
\begin{align*}
u(x,y)=
&\int_0^\infty\int_K \int_0^t p(t-s,z,y) \,\pi_{x,K}(\dd s,\dd z)\,\dd t\\
=&\int_0^\infty \int_K\int_s^\infty p(t-s,z,y)\,\dd t\,\pi_{x,K}(\dd s,\dd z)\\
=&\int_0^\infty \int_K u(z,y)\,\,\pi_{x,K}(\dd s,\dd z)\\
=&\int_K u(z,y)\,\gamma_{x,K}(\dd z),
\end{align*}
which proves the first assertion. The second one is a direct consequence of \eqref{o-1} and \eqref{o-2}.
\end{proof}

\begin{lemma}\label{hit-two}
There exist constants $c_1,c_2>0$ such that for any $x_0,x\in M$ with $d(x,x_0)\ge r$,
\begin{align*}
c_1 \frac{V(r)}{\phi(r)} \frac{\phi(d(x,x_0)+r)}{V(d(x,x_0)+r)}
&\le \Pp^x\big(d(X_t,x_0)\le r\textrm{ for some } t>0\big)\\
 &\le c_2 \frac{V(r)}{\phi(r)} \frac{\phi(d(x,x_0)-r)}{V(d(x,x_0)-r)}.\end{align*} \end{lemma}
\begin{proof}
Since for any $y\in B({x_0},r)$
$$d(x,y)\geq d(x,x_0)-d(x_0,y)\geq d(x,x_0)-r,$$
we have, by (\ref{ass-v-phi-1}) and Lemma \ref{lemma-1}, that for any $y\in B({x_0},r)$
\begin{equation}\label{green-upper}
u(x,y)\asymp \frac{\phi(d(x,y))}{V(d(x,y))}\lesssim \frac{\phi(d(x,x_0)-r)}{V(d(x,x_0)-r)}.
\end{equation}
Similarly, we obtain for any $y,z\in B(x_0, r)$,
$$u(z,y)\asymp \frac{\phi(d(z,y))}{V(d(z,y))}\gtrsim \frac{\phi(2r)}{V(2r)}\asymp \frac{\phi(r)}{V(r)}.$$
Hence the upper bound follows by applying (\ref{o-3}) to $K=\overline{B({x_0},r)}$.

For any $x\in M$ with $d(x,x_0)\ge r$, according to \eqref{hunt},
\begin{align*}
 \Pp^x(d(X_t,x_0)\le r\textrm{ for some } t>0)
&=\Pp^x(\sigma_{\overline{B(x_0,r)}}<\infty)\\
&= \int_{\overline{B(x_0,r)}} u(x,y)\,\nu_{\overline{B(x_0,r)}}(\dd y)\\
&\ge \inf_{y\in\overline{ B(x_0,r})}u(x,y) \nu_{\overline{B(x_0,r)}}(\overline{B(x_0,r)}).
\end{align*}
On the one hand, in a similar way to (\ref{green-upper}), we have
$$\inf_{y\in\overline{ B(x_0,r})}u(x,y) \gtrsim\frac{\phi(d(x,x_0)+r)}{V(d(x,x_0)+r)}.$$
On the other hand, by Lemma \ref{lemma-2},
$$\nu_{\overline{B(x_0,r)}}(\overline{B(x_0,r)})=\capa_{(0)}(\overline{B(x_0,r)})\gtrsim \frac{V(r)}{\phi(r)}.$$
Combining all the conclusions above, we complete the proof.
\end{proof}

In the following, we fix $x_0\in M$ and $t,r>0$, and define
$$Q(x,r,t)=\Pp^x(d(X_s,x_0)\le r\textrm{ for some }s>t).$$

\begin{proposition}\label{prop-1} There exists a constant $c_1>0$ such that for $x,x_0\in M$ with $d(x,x_0)\le r$ and for any $t\ge \phi(r)$,
$$Q(x,r,t)\le c_1 \frac{V(r)}{\phi(r)} \frac{t}{V(\phi^{-1}(t))}.$$   \end{proposition}
\begin{proof}
By the Markov property, we have  \begin{align*}
 Q(x,r,t)&=\int_M \Pp^y(d(X_s,x_0)\le r\textrm{ for some } s>0)\,\Pp^x(X_t\in \dd y)\\
 &\le \int_{d(y,x_0)-r\ge \phi^{-1}(t)} \Pp^y(d(X_s,x_0)\le r\textrm{ for some } s>0)\,\Pp^x(X_t\in \dd y)\\
 &\quad +   \int_{r\le  d(y,x_0)-r< \phi^{-1}(t)} \Pp^y(d(X_s,x_0)\le r\textrm{ for some } s>0)\,\Pp^x(X_t\in \dd y)\\
 &\quad +  \int_{d(y,x_0)-r< r} \Pp^y(d(X_s,x_0)\le r\textrm{ for some } s>0)\,\Pp^x(X_t\in \dd y)\\
 &=:I_1+I_2+I_3.
\end{align*}
For $I_1$, Lemma \ref{hit-two} implies that
\begin{align*}
I_1\le & c_1\frac{V(r)}{\phi(r)}\int_{d(y,x_0)-r\ge \phi^{-1}(t)} \frac{\phi(d(y,x_0)-r)}{ V(d(y,x_0)-r)} \frac{t}{V(d(x,y))\phi(d(x,y))}\,\mu(\dd y)\\
\le& c_2 \frac{V(r)}{\phi(r)}\int_{\phi^{-1}(t)+r}^\infty \frac{\phi(s-r)}{ V(s-r)} \frac{t}{V(s-d(x,x_0))\phi(s-d(x,x_0))}\,\dd V(s)\\
\le& c_2  \frac{V(r)}{\phi(r)}\int_{\phi^{-1}(t)+r}^\infty \frac{\phi(s)}{ V(s/2)} \frac{t}{V(s/2)
\phi(s/2)}\,\dd V(s)\\
\le& c_3 \frac{V(r)}{\phi(r)}\int_{\phi^{-1}(t)}^\infty \frac{t}{ V^2(s)}\,\dd V(s)\\
\le&c_4 \frac{V(r)}{\phi(r)} \frac{t}{ V(\phi^{-1}(t))},
\end{align*}
where the last inequality follows from the fact that
\begin{equation}\label{int-v}
\int_t^\infty \frac{1}{V^2(s)}\,\dd V(s)\asymp \frac{1}{V(t)},
\end{equation}
see the proof of \eqref{ass-v-phi}.
For $I_2$, Lemma \ref{hit-two} also implies that
\begin{align*}
I_2\le & c_1\frac{V(r)}{\phi(r)}
\int_{r\le d(y,x_0)-r< \phi^{-1}(t)} \frac{\phi(d(y,x_0)-r)}{ V(d(y,x_0)-r)} \frac{1}{V(\phi^{-1}(t))}\,\mu(\dd y)\\
\le& c_5 \frac{V(r)}{\phi(r)}\frac{1}{V(\phi^{-1}(t))}\int_r^{\phi^{-1}(t)} \frac{\phi(s)}{ V(s)} \,\dd V(s+r)\\
\le& c_6  \frac{V(r)}{\phi(r)} \frac{t}{ V(\phi^{-1}(t))},
\end{align*} where in the last inequality we have used the fact that
$$ \int_r^t \frac{\phi(s)}{ V(s)} \,\dd V(s+r) \le c_7  \phi(t), \quad t\ge r>0.$$
For $I_3$, we have
\begin{align*} I_3 \le &c_8 \int_{d(y,x_0)-r< r} \,\Pp^x(X_t\in \dd y)
=c_8 \int_{d(y,x_0)-r<r}p(t,x,y)\,\mu(\dd y)\\
\le& \frac{c_9 V(r)}{V(\phi^{-1}(t))} \le c_{10}  \frac{V(r)}{\phi(r)} \frac{t}{ V(\phi^{-1}(t))} .\end{align*}
Combining all the estimates above, we prove the desired assertion.
\end{proof}

\begin{remark}\label{r-prop-1}According to the proof above, it is easy to see that the statement of Proposition \ref{prop-1} still holds
for $$Q^*(x,r,t):=\Pp^x(d(X_s,x_0)\le r\textrm{ for some }s\ge t).$$

\end{remark}

Now we can present the

\begin{proof}[{\bf Proof of the One-Probability Statement in Theorem $\ref{lower-rate}$}]
For any $n\geq 0$, set $t_n=2^n$. Thanks to (\ref{scaling-cond}),
$\phi^{-1}(t_{n+1})\leq (2/c_3)^{1/d_3}\phi^{-1}(t_n)$. Then, by the
fact that $g(t)$ is strictly decreasing, we obtain
$$\varphi(s)=\phi^{-1}(s)g(s)\leq \phi^{-1}(t_{n+1})g(t_n)\leq c\phi^{-1}(t_n)g(t_n)=c\varphi(t_n),\quad s\in (t_n, t_{n+1}],$$  where
$c=(2/c_3)^{1/d_3}$. Next, for any $n\ge 1$, define
$$A_n=\big\{\text{$d(X_s, X_0)\leq c^{-1}\varphi(s)$ for some $s\in (t_n,t_{n+1}]$}\big\}.$$
Since $g(t)\to 0$ as $t \to\infty$, $\varphi(t)\le \phi^{-1}(t)$ for $t$ large enough.
Therefore, according to Proposition \ref{prop-1}, (\ref{volume-cond}) and (\ref{scaling-cond}), for $n\ge1$ large enough,
\begin{align*}
\Pp^x(A_n) \leq \Pp^x\left(\text{$d(X_s, x)\leq \varphi(t_n)$ for
some $s>t_n$}\right)
&\lesssim  \frac{V(\varphi(t_n))}{\phi(\varphi(t_n))}\frac{t_n}{V(\phi^{-1}(t_n))}\\
&\asymp\frac{V(\varphi(t_n))}{\phi(\varphi(t_n))}\frac{t_n-t_{n-1}}{V(\phi^{-1}(t_n))}.
\end{align*}
Then, by this inequality, (\ref{ass-v-phi-1}) and the fact that for
all $s\in(t_{n-1},t_n]$,
$$\varphi(s)=\phi^{-1}(s)g(s)\ge
\phi^{-1}(t_{n-1})g(t_n)\ge c'\varphi(t_n)$$
for some constant $c'>0$, we have
$$\sum_{n=1}^{\infty}\Pp^x(A_n)
\lesssim
\sum_{n=1}^{\infty}\frac{V(\varphi(t_n))}{\phi(\varphi(t_n))}\frac{t_n-t_{n-1}}{V(\phi^{-1}(t_n))}
\leq
c_1\int_1^{\infty}\frac{V(\varphi(t))}{\phi(\varphi(t))V(\phi^{-1}(t))}\,\dd
t.$$ Hence the Borel-Cantelli lemma yields that for all $x\in M$,
$$ \Pp^x(d(X_s, x)\ge c^{-1}\varphi(s)\textrm{ for all sufficient large }s)=1.$$
Replacing $c^{-1}\varphi$ with $\varphi$, we finally arrive at the required assertion for the one-probability statement.
\end{proof}

\ \

We consider below the zero-probability statement in Theorem \ref{lower-rate} for the lower rate function of the process $X$.

\begin{proposition} \label{prop-2} There is a constant $c_2>0$ such that for any $t\ge \phi(r)$ and for any $x,x_0\in M$ with $d(x,x_0)\le r$,
$$Q(x,r,t)\ge c_2 \frac{V(r)}{\phi(r)} \frac{t}{V(\phi^{-1}(t))}.$$\end{proposition}

\begin{proof}  By the Markov property and Lemma \ref{hit-two}, we have  \begin{align*}
 Q(x,r,t)&=\int_M \Pp^y(d(X_s,x_0)\le r\textrm{ for some } s>0)\,\Pp^x(X_t\in \dd y)\\
 &\ge c_1\frac{V(r)}{\phi(r)}\int_{d(y,x_0)\ge 2\phi^{-1}(t)} \frac{\phi(d(y,x_0)+r)}{V(d(y,x_0)+r)}  \frac{t}{V(d(x,y))\phi(d(x,y))} \,\mu(\dd y)  \\
 &\ge c_1 \frac{V(r)}{\phi(r)}\int_{d(y,x_0)\ge 2\phi^{-1}(t)} \frac{\phi(d(y,x_0)+r)}{V(d(y,x_0)+r)} \\
 &\qquad\qquad \qquad \quad  \times  \frac{t}{V(d(y,x_0)+d(x,x_0))\phi(d(y,x_0)+d(x,x_0))} \,\mu(\dd y)  \\
 &\ge c_2 \frac{V(r)}{\phi(r)}\int_{2\phi^{-1}(t)}^\infty  \frac{t}{V^2(s+r)} \,\dd V(s),
\end{align*} where in the first inequality we have used the fact that
$$d(x,y)\ge d(y,x_0)-d(x,x_0)\ge 2\phi^{-1}(t)-r\ge \phi^{-1}(t).$$
Since $\phi^{-1}(t)\geq r$, we have for any $s\geq 2\phi^{-1}(t)$,
$$V(s+r)\leq V(s+\phi^{-1}(t))\leq V(3s/2)\asymp V(s).$$
Then, by this inequality and (\ref{int-v}),
$$\frac{V(r)}{\phi(r)}\int_{2\phi^{-1}(t)}^\infty  \frac{t}{V^2(s+r)} \,\dd V(s)
\gtrsim
\frac{V(r)}{\phi(r)}\int_{2\phi^{-1}(t)}^\infty  \frac{t}{V^2(s)} \,\dd V(s)
\asymp \frac{V(r)}{\phi(r)} \frac{t}{V(\phi^{-1}(t))}.$$
The proof is complete.   \end{proof}

Furthermore, for fixed $x_0\in M$, $t,r>0$ and $\theta>1$, we define
$$R(x,r,t, \theta)=\Pp^x\big(d(X_s,x_0)\le r\textrm{ for some }s\in(t, \theta t]\big).$$
\begin{corollary}\label{cor-1} There exist constants $c_1,c_2>0$ such that for $t\ge \phi(r)$, $\theta\ge c_1$ and for $x,x_0\in M$ with $d(x,x_0)\le r$,
$$R(x,r,t,\theta)\ge c_2 \frac{V(r)}{\phi(r)} \frac{t}{V(\phi^{-1}(t))}.$$\end{corollary} \begin{proof} By the definition, we have
$$R(x,r,t,\theta)= Q(x,r,t)- Q(x,r,\theta t).$$ This along with Propositions \ref{prop-1} and \ref{prop-2} yields that
$$R(x,r,t,\theta)\ge c_2 \frac{V(r)}{\phi(r)} \frac{t}{V(\phi^{-1}(t))}- c_1\frac{V(r)}{\phi(r)} \frac{\theta t}{V(\phi^{-1}(\theta t))}.$$
By using  \eqref{vp} and taking $\theta>1$ such that
$$\frac{c_1}{c_0c_2} \theta^{1-d_1/d_4}\le\frac{1}{2}$$ (here $c_0>0$ is the constant $c_0$ in \eqref{vp}), we arrive at that
$$R(x,r,t,\theta)\ge \frac{c_2}{2} \frac{V(r)}{\phi(r)} \frac{t}{V(\phi^{-1}(t))}.$$ The proof is complete.
 \end{proof}

We need the following Borel-Cantelli lemma taken from \cite[Lemma B]{Tak2}, which is a simplification of \cite[Theorem 1]{CE}.
\begin{lemma}\label{0-1-lemma} Let $(A_k)_{k\ge1}$ be a sequence of events satisfying the following three conditions:
\begin{itemize}
\item[(i)] $$\sum_{k=1}^\infty \Pp(A_k)=\infty.$$
\item[(ii)] $$\Pp(\limsup A_k)=0\textrm{ or }1.$$
\item[(iii)] There exist two constants $c_1,c_2>0$ with the following property{\rm:} to each $A_j$ there corresponds a set of events $A_{j_1}, \cdots, A_{j_s}$ belonging to $\{A_k\}_{k\ge1}$ such that
    $$\sum_{i=1}^s \Pp(A_j\cap A_{j_i})\le c_1 \Pp(A_j)$$ and that for other $A_i$ than $A_{j_i}$ $(1\le i\le s)$
which stands after $A_j$ in the sequence $(A_k)_{k\ge1}$
{\rm (}viz.\ $i>j${\rm )},
the inequality
    $$\Pp(A_j\cap A_i)\le c_2\Pp(A_j)\Pp(A_i)$$ holds.
\end{itemize}
Then, infinity many events $(A_k)_{k\ge1}$ occur with probability $1$.
\end{lemma}

The following result has been proved in \cite[Theorem 2.10]{KKW}.

\begin{lemma}[\bf{The Zero-One Law for Tail Events}]\label{t:01tail}
Let $p(t,x,y)$ satisfy \eqref{hkestimate} as above, and let $A$ be a
tail event. Then, either $\Pp^x(A)$ is $0$ for all $x$ or else it is
$1$ for all $x \in M$.
\end{lemma}

Finally, we are in a position to present the
\begin{proof}[{\bf Proof of the Zero-Probability Statement in Theorem $\ref{lower-rate}$}]
By \eqref{scaling-cond} and the fact that $g(t)\to 0$ as $t
\to\infty$, we can without loss of generality assume that $g(t)\le
1$ and $\varphi(t)\le \phi^{-1}(\kappa t)$ for all $t>0$ and some
constant $\kappa\in(0,1)$. Let $\theta>1$ be the constant in
Corollary \ref{cor-1} such that $1-1/\theta\ge \kappa$. Define
$$A_n= \big\{d(X_s, X_0))\le c\varphi (\theta^{n+1})\textrm{ for some }s\in (\theta^n, \theta^{n+1}]\big\},$$
where $c\in(0,1)$ satisfies that $c\phi^{-1}(\theta^{n+1})\le
\phi^{-1}(\theta^n)$ for all $n\geq 1$. This implies that for any
$s\in [\theta^{n+1}, \theta^{n+2}]$,
\begin{equation}\label{comp-phi}
\varphi(\theta^{n+1})=\phi^{-1}(\theta^{n+1})g(\theta^{n+1})
\geq c\phi^{-1}(\theta^{n+2})g(s)\geq c\phi^{-1}(s)g(s)=c\varphi(s).
\end{equation}
Then, according to Corollary \ref{cor-1} and (\ref{ass-v-phi-1}), there is a constant $c_0>0$ such that
\begin{align*}
\sum_{n=1}^\infty\Pp^x(A_n)&=\sum_{n=1}^\infty R(x,c\varphi (\theta^{n+1}) , \theta^n,\theta)\\
&\ge c_0\sum_{n=1}^\infty \frac{V(c\varphi (\theta^{n+1}))}{\phi(c\varphi (\theta^{n+1}))}  \frac{\theta^n}{V(\phi^{-1}(\theta^n))}\\
&\asymp \sum_{n=1}^\infty \frac{V(c\varphi (\theta^{n+1}))}{\phi(c\varphi (\theta^{n+1}))}  \frac{\theta^{n+2}-\theta^{n+1}}{V(\phi^{-1}(\theta^{n+1}))}\\
&\gtrsim \int_1^{\infty}\frac{V(\varphi(t))}{\phi(\varphi(t))V(\phi^{-1}(t))}\,\dd t.
\end{align*}

We define a sequence of stopping times $\{\sigma_n\}_{n\ge1}$ by
$$\sigma_n=\inf\big\{ t\in (\theta^n,\theta^{n+1}]: d(X_t, X_0)\le c\varphi(\theta^{n+1})\big\},$$ where $\inf \emptyset=\infty.$ As mentioned above,
we assume that $\varphi(t)\le \phi^{-1}(\kappa t)$ for all $t>0$.
By the strong Markov property,  for any $i\ge j+2$,
\begin{align*}
&\Pp^x(A_i\cap A_j)\\
&=\Pp^x(\sigma_j\leq \theta^{j+1}, \sigma_i\leq \theta^{i+1})\\
&=\Ee^x\Big(\Pp^{X_{\sigma_j}}\left(\text{$d(X_t,x)\leq c\varphi(\theta^{i+1})$
for some $t\in (\theta^i-s,\theta^{i+1}-s]$}\right)\Big|_{s={\sigma_j}}; \sigma_j\leq \theta^{j+1}\Big)\\
&\leq \Ee^x \Big(\Pp^{X_{\sigma_j}}
\left(\text{$d(X_t,x)\leq c\varphi(\theta^{i+1})$ for some $t>\theta^i-\theta^{j+1}$}\right); \sigma_j\leq \theta^{j+1}\Big)\\
&\le \Pp^x( \sigma_j\leq \theta^{j+1})
\sup_{d(z,x)\le c\varphi(\theta^{j+1})} \Pp^z\left(d(X_t,x)\le c\varphi(\theta^{i+1})\textrm{ for some }t>\theta^i-\theta^{j+1}\right).
\end{align*}
Note that, for any $z\in M$,
\begin{align*}
&\Pp^z\left(d(X_t,x)\le c\varphi(\theta^{i+1})\textrm{ for some }t\ge\theta^i-\theta^{j+1}\right)\\
&=\int_{\overline{B(x,c\varphi(\theta^{i+1}))}}p(\theta^i-\theta^{j+1},z,y)\,\mu({\rm d}y)\\
&\quad
+\int_{\overline{B(x,c\varphi(\theta^{i+1}))}^c}p(\theta^i-\theta^{j+1},z,y)
\Pp^y(d(X_t,x)\le c\varphi(\theta^{i+1})\textrm{ for some
}t>0)\,\mu({\rm d}y).
\end{align*}
Since for all $i\geq j+2$,
$$\phi(c\varphi(\theta^{i+1}))\leq \phi(\varphi(\theta^i))\leq \phi(\phi^{-1}(\kappa\theta^i))
=\kappa \theta^i\leq \theta^i-\theta^{i-1} \leq
\theta^{i}-\theta^{j+1},$$ we get from Lemma \ref{comp-heat} that,
if $d(z,x)\leq c\varphi(\theta^{j+1})$, then
\begin{align*}
&\Pp^z\left(d(X_t,x)\le c\varphi(\theta^{i+1})\textrm{ for some }t\ge\theta^i-\theta^{j+1}\right)\\
&\asymp \Pp^x\left(d(X_t,x)\le c\varphi(\theta^{i+1})\textrm{ for
some }t\ge\theta^i-\theta^{j+1}\right).
\end{align*}
On the other hand, also due to
$$\phi(c\varphi(\theta^{i+1}))\leq \theta^{i}-\theta^{j+1}, \quad i\ge j+2,$$
Proposition \ref{prop-1} and Remark \ref{r-prop-1} imply that
\begin{align*}
\Pp^x\left(d(X_t,x)\le c\varphi(\theta^{i+1})\textrm{ for some
}t\ge\theta^i-\theta^{j+1}\right)
&\lesssim \frac{V(c\varphi(\theta^{i+1}))}{\phi(c\varphi(\theta^{i+1}))}\frac{\theta^i-\theta^{j+1}}{V(\phi^{-1}(\theta^i-\theta^{j+1}))}\\
&\lesssim
\frac{V(\varphi(\theta^i))}{\phi(\varphi(\theta^i))}\frac{\theta^i-\theta^{j+1}}{V(\phi^{-1}(\theta^i-\theta^{j+1}))},
\end{align*}
where the last inequality follows from \eqref{ass-v-phi-1} and
\eqref{comp-phi}.

Since
$$\frac{V(\phi^{-1}(R))}{V(\phi^{-1}(r))} \le c_0' \left( \frac{R}{r}\right)^{d_2/d_3},\quad R\ge r>0,$$ we have
$$V(\phi^{-1}(\theta^i-\theta^{j+1}))\ge c_1 V(\phi^{-1}(\theta^i)),\quad i\ge j+2,$$
which along with Proposition \ref{prop-2} yields that there exists a constant $c_2>0$ such that for all $i\ge j+2$,
$$\Pp^x(A_i\cap A_j)\le c_2 \Pp^x(A_i)\Pp^x(A_j).$$ On the other hand, we always have
$$\Pp^x(A_{j+1}\cap A_j)\le \Pp^x(A_j).$$
Furthermore, according to Lemma \ref{t:01tail}, we have $$\Pp(\limsup A_k)=0\textrm{ or }1.$$

Combining all the conclusions above with Lemma \ref{0-1-lemma} and the fact that
\begin{align*} A_k\subset \big\{ d(X_s,X_0)\le \varphi(s)\textrm{ for some }s\in(\theta^k, \theta^{k+1}] \big\},\end{align*}  we prove the desired assertion.
 \end{proof}

\subsection{Critical  Case}
In this subsection, we consider lower rate functions
under the following critical condition:
\begin{assumption} \label{critical}
\emph{The constants $d_i \ (i=1,2,3,4)$ in $(\ref{volume-cond})$ and $(\ref{scaling-cond})$
satisfy
$$d_1=d_2=d_3=d_4.$$}
\end{assumption}
Our main contribution in this part is
\begin{theorem}\label{critical-1}
Let $g(t)$ be a strictly positive function on $(0,\infty)$ such that
$g(t)\searrow0$ as $t\rightarrow\infty$.
Under  Assumption {\rm \ref{critical}}, if
$$\int_1^{\infty}\frac{1}{t|\log g(t)|}\,\dd t<\infty \quad (\text{{\it resp.} $=\infty$}),$$
then the function $\varphi(t)=\phi^{-1}\left(tg(t)\right)$ satisfies that for all $x_0\in M$,
$$\Pp^{x_0}(\text{$d(X_s,x_0)\geq \varphi(s)$ for all sufficiently large $s$})=1\quad (\text{{\it resp.} $=0$}).$$
\end{theorem}

Since Assumption {\rm \ref{critical}} implies that $X$ is recurrent
and can not hit a point (see Proposition \ref{recurrence} below), the
rate function in Theorem \ref{critical-1} describes how $X$ comes
close to the starting point for all sufficiently large time. Spitzer
established in \cite{S} an integral test on the zero-one law of the
lower rate functions for the two-dimensional Brownian motion.
Takeuchi and Watanabe \cite{TW}  extended this test to the
one-dimensional Cauchy process, and Khoshnevisan \cite{K} also
extended it to the direct product of stable processes. Theorem
\ref{critical-1} is an extension of \cite{S,TW}
to symmetric Markov processes on general state spaces. According to
Theorem \ref{critical-1}, the function
$$\varphi(t)=\phi^{-1}\left(\frac{t}{\exp{((\log t)(\log\log t)^{1+\varepsilon})}}\right), \quad \varepsilon\ge -1$$
is a lower rate function for $X$, if and only if $\varepsilon>0$.

\ \

To prove Theorem \ref{critical-1}, we first follow the argument of Khoshnevisan \cite{K} and obtain some key probability estimates.
For any $x_0\in M$, $r>0$ and $0<a<b$, set
$$\Phi(x_0,r,a,b)=\Pp^{x_0}\big(d(X_s,x_0)\le r\textrm{ for some }s\in(a, b]\big).$$

\begin{lemma}\label{hit-interval}
 For any $x_0\in M$, $r>0$ and $0<a<b$,
it holds that
\begin{equation*}
\begin{split} \frac{\int_a^b \Pp^{x_0}\left(d(X_u,x_0)\le r\right)\,\dd u}{2\int_0^{b-a}
\Big(\sup\limits_{d(x,x_0)\le r}\Pp^x(d(X_{u},x)\le 2r)\Big)\,\dd u}
&\le \Phi(x_0,r,a,b)\\
&\le \frac{ \int_a^{2b-a} \Pp^{x_0}\left(d(X_u,x_0)\le r\right)\,\dd u}
{  \int_0^{b-a}\Big( \inf\limits_{d(x,x_0)\le r} \Pp^{x}(d(X_u,x)\le r)\Big)\,\dd u}.
\end{split}
\end{equation*}
\end{lemma}

\begin{proof} For any $r>0$ and $b>a>0$, define $$T=T_{r,a,b}=\inf\{s\in(a,b]: d(X_s,x_0)\le r\}.$$
Applying the strong Markov property at time $T$, we see that
\begin{equation*}
\begin{split}
 \Ee^{x_0}\left( \int_a^{2b-a} \I_{\{d(X_u,x_0)\le 2r\}}\,\dd u \right)
=& \int_a^{2b-a} \Pp^{x_0}(d(X_u,x_0)\le 2r)\,\dd u\\
 = &  \int_a^{2b-a} \Pp^{x_0}(d(X_u,x_0)\le 2r\big| T\le b)\,\,\Pp^{x_0}(T\le b)\,\dd u\\
 \ge & \int_0^{b-a} \left(  \inf_{d(x,x_0)\le r} \Pp^{x} (d(X_u,x)\le r)\right)\,\dd u \,\, \Pp^{x_0}(T\le b).
\end{split}
\end{equation*} That is, we have
$$\Pp^{x_0}(T\le b)\le  \frac{\Ee^{x_0}\left( \int_a^{2b-a} \I_{\{d(X_u,x_0)\le 2r\}}\,\dd u \right)}{  \int_0^{b-a}\left( \inf_{d(x,x_0)\le r} \Pp^{x}(d(X_u,x)\le r)\right)\,\dd u},$$ which yields the upper bound.

On the other hand, using the Markov property,
\begin{equation*}
\begin{split}
 &\Ee^{x_0}\left(\int_a^b \I_{\{d(X_u,x_0)\le r\}}\,\dd u\right)^2\\
 &=2\Ee^{x_0}\left(\int_a^b\int_a^u \I_{\{d(X_u,x_0)\le r\}}\I_{\{d(X_v,x_0)\le r\}}\,\dd v\,\dd u\right)\\
 &\le 2 \Ee^{x_0}\left(\int_a^b\int_a^u \I_{\{d(X_u,X_v)\le 2r\}}\I_{\{d(X_v,x_0)\le r\}}\,\dd v\,\dd u\right)\\
 &\le 2\int_a^b\int_a^u \Pp^{x_0} (d(X_v,x_0)\le r)\sup_{d(x,x_0)\le r}\Pp^x(d(X_{u-v},x)\le 2r)\,\dd v\,\dd u\\
 &\le 2\int_a^b \Pp^{x_0} (d(X_u,x_0)\le r)\,\dd u \int_0^{b-a} \sup_{d(x,x_0)\le r}\Pp^x(d(X_{u},x)\le 2r)\,
 \dd u\\
 &= 2 \Ee^{x_0}\left(\int_a^b \I_{\{d(X_u,x_0)\le r\}}\,\dd u\right)\int_0^{b-a} \sup_{d(x,x_0)\le r}\Pp^x(d(X_{u},x)\le 2r)\,\dd u.
\end{split}
\end{equation*}
According to the Cauchy-Schwarz inequality,
\begin{equation*}
\begin{split}
\Ee^{x_0}\left(\int_a^b \I_{\{d(X_u,x_0)\le r\}}\,\dd u\right) &=\Ee^{x_0} \left(\int_a^b \I_{\{d(X_u,x_0)\le r\}}\,\dd u; T\le b\right) \\
 &\le \sqrt{\Ee^{x_0}\left(\int_a^b \I_{\{d(X_u,x_0)\le r\}}\,\dd u\right)^2} \sqrt{\Pp^{x_0}(T\le b)}.
\end{split}
\end{equation*}
Hence, by these inequalities above, we arrive at
\begin{equation*}
\begin{split}
&\Ee^{x_0}\left(\int_a^b \I_{\{d(X_u,x_0)\le r\}}\,\dd u\right)\\
\le  &\sqrt{2} \sqrt{\Ee^{x_0}\left(\int_a^b \I_{\{d(X_u,x_0)\le r\}}\,\dd u\right)}\sqrt{\int_0^{b-a} \sup_{d(x,x_0)\le r}\Pp^x(d(X_{u},x)\le 2r)\,\dd u}\\
&\times \sqrt{\Pp^{x_0}(T\le b)}.
\end{split}
\end{equation*} Therefore,
$$\Pp^{x_0}(T\le b)\ge \frac{\Ee^{x_0}\left(\int_a^b \I_{\{d(X_u,x_0)\le r\}}\,\dd u\right)}{2\int_0^{b-a}\left(\sup_{d(x,x_0)\le r}\Pp^x(d(X_{u},x)\le 2r)\right)\,\dd u}.$$ This completes the proof.
\end{proof}

\begin{lemma}\label{hit-ball} For any $r, t>0$ and for all $x\in M$,
$$\Pp^x(d(X_t,x)\le r) \asymp 1\wedge \frac{ V(r)}{V(\phi^{-1}(t))}.$$ \end{lemma}
\begin{proof} The conclusion directly follows from
$$\Pp^x(d(X_t,x)\le r)=\int_{d(y,x)\le r} p(t,x,y)\,\mu(\dd y)$$ and \eqref{hkestimate}.
\end{proof}

\begin{remark}\label{growth-rate}
Under (\ref{volume-cond}) and (\ref{scaling-cond}),
Assumption \ref{critical} implies that
$$V(r)\asymp \phi(r)\asymp r^{d_1},$$
and thus
$$\int_{\alpha}^{\beta}\frac{1}{V(\phi^{-1}(u))}\,\dd u
\asymp \int_{\alpha}^{\beta}\frac{1}{u}\,\dd u=\log\left(\frac{\beta}{\alpha}\right)$$
for any $0<\alpha<\beta$.
\end{remark}

\begin{lemma}\label{hit-0}
Let Assumption {\rm \ref{critical}} hold.
If $\phi(r)\leq b-a$, then
$$
\Pp^{x_0}(\text{$d(X_s,x_0)\leq r$ for some $s\in (a,b]$})
\lesssim \frac{(\phi(r)-a)_+ +\phi(r)\log\left(\frac{2b-a}{a\vee \phi(r)}\right)}
{\phi(r)\left\{1+\log\left(\frac{b-a}{\phi(r)}\right)\right\}}
$$
and
$$
\Pp^{x_0}(\text{$d(X_s,x_0)\leq r$ for some $s\in (a,b]$})
\gtrsim \frac{(\phi(r)-a)_+ +\phi(r)\log\left(\frac{b-a}{a\vee \phi(r)}\right)}
{\phi(r)\left\{1+\log\left(\frac{b-a}{\phi(r)}\right)\right\}}.
$$

\end{lemma}
\begin{proof}
By Remark \ref{growth-rate},
\begin{equation*}
\begin{split}
\int_a^{2b-a}\left(1\wedge \frac{V(r)}{V(\phi^{-1}(u))}\right)\,\dd u
&=\int_a^{a\vee \phi(r)}\,\dd u+\int_{a\vee\phi(r)}^{2b-a}\frac{V(r)}{V(\phi^{-1}(u))}\,\dd u\\
&\asymp (\phi(r)-a)_+ +V(r)\log\left(\frac{2b-a}{a\vee\phi(r)}\right)\\
&\asymp (\phi(r)-a)_+ +\phi(r)\log\left(\frac{2b-a}{a\vee \phi(r)}\right)
\end{split}
\end{equation*}
and
\begin{equation}\label{int-p}
\int_0^{b-a}\left(1\wedge \frac{V(r)}{V(\phi^{-1}(u))}\right)\,\dd u
\asymp  \phi(r)\left\{1+\log\left(\frac{b-a}{\phi(r)}\right)\right\}.
\end{equation}
Hence, according to Lemmas \ref{hit-interval} and \ref{hit-ball},
\begin{equation*}
\begin{split}
\Pp^{x_0}(\text{$d(X_s,x_0)\leq r$ for some $s\in (a,b]$}) &\lesssim
\frac{\int_a^{2b-a}\left(1\wedge
\frac{V(r)}{V(\phi^{-1}(u))}\right)\,\dd u}
{\int_0^{b-a}\left(1\wedge \frac{V(r)}{V(\phi^{-1}(u))}\right)\,\dd u}\\
&\asymp
\frac{(\phi(r)-a)_+ +\phi(r)\log\left(\frac{2b-a}{a\vee \phi(r)}\right)}
{\phi(r)\left\{1+\log\left(\frac{b-a}{\phi(r)}\right)\right\}}.
\end{split}
\end{equation*}
We can get the lower bound by the same way.
\end{proof}

\begin{remark}\label{r-hit-0} The statements of Lemma \ref{hit-0} still hold for
$$\Pp^{x_0}(\text{$d(X_s,x_0)\leq r$ for some $s\in [a,b]$}).$$
\end{remark}

\begin{proposition} \label{recurrence}
Under Assumption {\rm \ref{critical}}, $X$ is recurrent and can not hit any point
from every starting point.
\end{proposition}
\begin{proof}
In a similar way to Lemma \ref{lemma-1},
we get from
Remark \ref{growth-rate} that
$$\int_0^{\infty}p(t,x,y)\,\dd t=\infty.$$
Therefore, by Remark \ref{recurrence-test}, $X$  is recurrent.

If $\phi(r)\leq a$, then Lemma \ref{hit-0} implies that
$$\Pp^{x_0}(\text{$d(X_s,x_0)\leq r$ for some $s\in (a,b]$})\lesssim
\frac{\log\left(\frac{2b-a}{a}\right)}{1+\log\left(\frac{b-a}{\phi(r)}\right)}.
$$
Hence, by letting first  $r\rightarrow 0$ and then $a\rightarrow 0$ and $b\rightarrow \infty$, we get
$$\Pp^{x_0}(\text{$X_s=x_0$ for some $s>0$})=0.$$
On the other hand, by the Markov property and the fact that the heat kernel $p(t,x,y)$ is strictly positive
for any $t>0$ and $x,y\in M$, it holds that
\begin{align*}
0
&=\Pp^{x_0}(\text{$X_s=x_0$ for some $s>a$})
=\Ee^{x_0}\left(\Pp^{X_a}(\text{$X_s=x_0$ for some $s>0$})\right)\\
&=\int_M p(a,x_0,y)\Pp^y(\text{$X_s=x_0$ for some $s>0$})\,\mu(\dd y),
\end{align*}
and so we obtain
$$\Pp^x(\text{$X_s=x_0$ for some $s>0$})=0 \quad \text{for $\mu$-a.e.\ $x\in M$}.$$
This shows that
\begin{align*}
\Pp^{x}(\text{$X_s=x_0$ for some $s>a$})
=\int_M p(a,x,y)\Pp^y(\text{$X_s=x_0$ for some $s>0$})\,\mu(\dd y)
=0
\end{align*}
for any $x\in M$. We thus complete the proof
by letting $a\rightarrow0$.
\end{proof}

Finally, we will present the

\begin{proof}[Proof of Theorem $\ref{critical-1}$] (i) (\textbf{One-probability statement})\,\,
For the simplification of the proof, we assume that $g(t)\le 1/2$ for any $t>0$.
Let $t_n=2^n$ for $n\geq 1$.
Since
$$\phi(\phi^{-1}(t_{n+1}g(t_n)))=t_{n+1}g(t_n)\leq t_n=t_{n+1}-t_n,$$
Lemma \ref{hit-0} shows that
\begin{equation*}
\begin{split}
&\Pp^{x_0}(\text{$d(X_s,x_0)\leq \varphi(s)$ for some $s\in (t_n,t_{n+1}]$})\\
&\leq
\Pp^{x_0}\left(\text{$d(X_s,x_0)\leq \phi^{-1}\left(t_{n+1}g(t_n)\right)$ for some $s\in (t_n,t_{n+1}]$}\right)\\
&\lesssim
\frac{\log\left(\frac{2t_{n+1}-t_n}{t_n}\right)}{1+\log\left(\frac{t_{n+1}-t_n}{\phi(\phi^{-1}\left(t_{n+1}g(t_n)\right))}\right)}
\asymp \frac{1}{1+\log\left(\frac{t_{n+1}-t_n}{t_{n+1}g(t_n)}\right)}
\lesssim \frac{1}{|\log g(t_n)|}.
\end{split}
\end{equation*}
Therefore,
\begin{equation*}
\begin{split}
\sum_{n=1}^{\infty}\Pp^{x_0}(\text{$d(X_s,x_0)\leq \varphi(s)$ for
some $s\in (t_n,t_{n+1}]$})
&\lesssim \sum_{n=1}^{\infty}\frac{t_n-t_{n-1}}{t_n|\log g(t_n)|}\\
&\lesssim \int_1^{\infty}\frac{1}{t|\log g(t)|}\,\dd t<\infty.
\end{split}
\end{equation*}
Hence, we finish the proof of the one-probability statement by the Borel-Cantelli lemma.

(ii) (\textbf{Zero-probability statement})\,\, In this part, we also assume that $g(t)\le 1/2$ for all $t>0$.
Let $t_n=2^n$ for $n\geq 1$. Define
$$A_n=\left\{\text{$d(X_s,x_0)\leq c\varphi(t_{n+1})$ for some $s\in [t_n,t_{n+1}]$}\right\},$$
where $c\in (0,1)$ satisfies that $c\varphi(t_{n+1})\le
\phi^{-1}(t_ng(t_{n+1}))$ for all $n\ge 1$. Since this inequality
implies that
$$\phi(c\varphi(t_{n+1}))\le {t_n}{g(t_{n+1})}\leq t_n=t_{n+1}-t_n,$$
we have by Lemma \ref{hit-0},
$$\Pp^{x_0}(A_n)\gtrsim \frac{\log\left(\frac{t_{n+1}-t_n}{t_n}\right)}{1+\log \left(\frac{t_{n+1}-t_n}{\phi(c\varphi(t_{n+1}))}\right)}
\gtrsim \frac{1}{|\log g(t_{n+1})|}.$$
Hence, in a similar way as in (i), we have
$$\sum_{n=1}^{\infty}\Pp^{x_0}(A_n)\gtrsim \int_1^{\infty}\frac{1}{t|\log g(t)|}\,\dd t=\infty.$$

Next, we define a sequence of stopping times $\{\sigma_n\}_{n\ge1}$ by
$$\sigma_n=\inf\big\{ t\in [t_n,t_{n+1}]: d(X_t, X_0)\le c\varphi(t_{n+1})\big\},$$ where $\inf \emptyset=\infty.$
By the strong Markov property,  for any $i\ge j+2$,
\begin{align*}
&\Pp^x(A_i\cap A_j)\\
&=\Pp^x(\sigma_j\leq t_{j+1}, \sigma_i\leq t_{i+1})\\
&=\Ee^x\Big(\Pp^{X_{\sigma_j}}\left(\text{$d(X_t,x)\leq c\varphi(t_{i+1})$
for some $t\in [t_i-s,t_{i+1}-s]$}\right)\Big|_{s={\sigma_j}}; \sigma_j\leq t_{j+1}\Big)\\
&\le \Pp^x( \sigma_j\leq t_{j+1})\sup_{d(z,x)\le c\varphi(t_{j+1})} \Pp^z\left(d(X_t,x)\le c\varphi(t_{i+1})\textrm{ for some }t\in[t_i- t_{j+1}, t_{i+1}] \right).
\end{align*}
For any $z\in M$,
\begin{align*} &\Pp^z\left(d(X_t,x)\le c\varphi(t_{i+1})\textrm{ for some }t\in[t_i- t_{j+1}, t_{i+1}] \right)\\
&=\int_{\overline{B(x,c\varphi(t_{i+1}))}} p(t_i-t_{j+1}, z,z_1)\,\mu(\dd z_1)\\
&\quad + \int_{{\overline{B(x,c\varphi(t_{i+1}))}}^c} p(t_i-t_{j+1}, z,z_1) \\
& \qquad \quad \times \Pp^{z_1}\big(d(X_s,x)\leq
c\varphi(t_{i+1})\textrm{ for some  }s\in (0,
t_{i+1}-t_i+t_{j+1}]\big)\,\mu(\dd z_1).
\end{align*}
Since $c\varphi(t_{j+1})\le \phi^{-1}(t_i-t_{j+1})$ for all $i\ge
j+2$, we find from Lemma \ref{comp-heat} that for any $z\in M$  with
$d(z,x)\le c\varphi(t_{j+1})$,
$$ p(t_i- t_{j+1}, z,z_1)\asymp p(t_i- t_{j+1}, x,z_1)$$ and so
\begin{align*}&\Pp^z\left(d(X_t,x)\le c\varphi(t_{i+1})\textrm{ for some }t\in[t_i- t_{j+1}, t_{i+1}] \right)\\
&\asymp \Pp^x\left(d(X_t,x)\le c\varphi(t_{i+1})\textrm{ for some
}t\in[t_i- t_{j+1}, t_{i+1}] \right).\end{align*} This, along with
Lemma \ref{hit-0} and Remark \ref{r-hit-0}, further yields that for
all $i\ge j+2$,
\begin{align*}
&\sup_{d(z,x)\le c\varphi(t_{j+1})} \Pp^z\left(d(X_t,x)\le c\varphi(t_{i+1})\textrm{ for some }t\in[t_i- t_{j+1}, t_{i+1}] \right)\\
&\lesssim  \Pp^x\left(d(X_t,x)\le c\varphi(t_{i+1})\textrm{ for some }t\in[t_i- t_{j+1}, t_{i+1}] \right) \\
&\lesssim \bigg({\log \frac{2t_{i+1}-t_i+t_{j+1}}{t_i-t_{j+1}}}\bigg)\bigg( {1+ \log \frac{t_{i+1}-t_i+t_{j+1}}{\phi(c\varphi(t_{i+1}))}}\bigg)^{-1}\\
&\lesssim \bigg({\log \frac{t_{i+1}-t_i}{t_i}}\bigg)\bigg({1+ \log \frac{t_{i+1}-t_i}{\phi(c\varphi(t_{i+1}))}}\bigg)^{-1}\\
&\lesssim \Pp^x( \sigma_i\leq t_{i+1}).\end{align*}
Therefore, there is a constant $C>0$ such that for any $i\ge j+2$,
$$\Pp^x(A_i\cap A_j)\le C\Pp^x(A_i)\Pp^x(A_j).$$

Having both conclusions above at hand, we can follow the argument of
the zero-probability statement in Theorem \ref{lower-rate} and
obtain
 $$\Pp(\limsup A_k)=1.$$
Hence, the desired assertion follows from the fact that
\begin{align*} A_k\subset \big\{ d(X_s,X_0)\le \varphi(s)\textrm{ for some }s\in[t_k, t_{k+1}] \big\}.\end{align*}
The proof is complete.
 \end{proof}

 \ \

\begin{ack}The research of Yuichi Shiozawa is supported in part by the Grant-in-Aid for Scientific Research (C) 26400135.
 The research of Jian Wang is supported by National
Natural Science Foundation of China (No.\ 11201073 and 11522106), the JSPS postdoctoral fellowship
(26$\cdot$04021), National Science Foundation of Fujian Province (No.\ 2015J01003), and the Program for Nonlinear Analysis and Its Applications (No. IRTL1206). \end{ack}

\end{document}